\documentclass[11pt, a4paper]{amsart}
\usepackage[dvips]{epsfig}
\usepackage{amsmath}
\usepackage{amssymb}
\usepackage{amsfonts}
\usepackage{amsthm}
\usepackage{amsbsy}
\usepackage{amsgen}
\usepackage{amscd}
\usepackage{amsopn}
\usepackage{amstext}
\usepackage{amsxtra}
\usepackage[utf8]{inputenc}
\usepackage{enumerate}
\usepackage{hyperref}
\usepackage{lineno}
\usepackage{marginnote}
\usepackage{verbatim}

\usepackage{times, graphicx}
\usepackage{eso-pic}
\usepackage{lastpage}

\numberwithin{equation}{section}

\newtheorem{theorem}{Theorem}[section]
\newtheorem{proposition}[theorem]{Proposition}
\newtheorem{lemma}[theorem]{Lemma}
\newtheorem{corollary}[theorem]{Corollary}
\newtheorem{example}[theorem]{Example}

\theoremstyle{definition}
\newtheorem{definition}[theorem]{Definition}
\newtheorem{notation}[theorem]{Notation}
\newtheorem{remark}[theorem]{Remark}

\newcommand{\vol}{\operatorname{Vol}}

\newcommand{\Dc}{\mathcal{D}}

\newcommand{\Sc}{\mathcal{S}}

\newcommand{\Ec}{\mathcal{E}}

\newcommand{\Bc}{\mathcal{B}}

\newcommand{\Fc}{\mathcal{F}}

\newcommand{\meas}{\operatorname{meas}}

\newcommand {\E} {\mathbb{E}}
\newcommand {\M} {\mathcal{M}}
\newcommand {\R} {\mathbb{R}}

\newcommand {\Z} {\mathbb{Z}}

\newcommand {\Ac} {\mathcal{A}}
\newcommand {\Uc} {\mathcal{U}}

\newcommand {\Tb} {\mathbb{T}}

\newcommand {\Vc} {\mathcal{V}}

\author[I. Wigman]{Igor Wigman}
\address{IW: Department of Mathematics \\ King's College London \\ Strand \\ London WC2R 2LS \\ England, UK}
\email{{\tt igor.wigman@kcl.ac.uk}}

\author[N. Yesha]{Nadav Yesha}
\address{NY: Department of Mathematics \\ King's College London \\ Strand \\ London WC2R 2LS \\ England, UK}
\curraddr{Department of Mathematics, University of Haifa, 3498838 Haifa, Israel}
\email{{\tt nyesha@univ.haifa.ac.il}}

\begin{document}
	
\title[CLT for mass distribution of toral Laplace eigenfunctions]{CLT for Planck scale mass distribution of toral Laplace eigenfunctions}

\begin{abstract}
We study the fine scale $L^{2}$-mass distribution of toral Laplace eigenfunctions with respect to random position, in $2$ and $3$ dimensions.
In $2$d, under certain flatness assumptions on the Fourier coefficients and generic restrictions on energy levels,
both the asymptotic shape of the variance is determined and the limiting Gaussian law is established, in the optimal Planck-scale regime.
In $3$d the asymptotic behaviour of the variance is analysed in a more restrictive scenario (``Bourgain's eigenfunctions").
Other than the said precise results, lower and upper bounds are proved for the variance, under more general flatness assumptions
on the Fourier coefficients.
\end{abstract}

\maketitle

\section{Introduction}

Given a smooth compact $d$-manifold $\M$ we are interested in the spectral properties of the Laplace-Beltrami operator $\Delta$ on $\M$.
It is well-known that the eigenvalue spectrum of $\Delta$ is purely discrete, i.e.,
the set of numbers $E$ admitting a solution to the Helmholtz equation
\begin{equation*}
\Delta \phi + E\phi = 0
\end{equation*}
is a sequence $\{E_{j}\}_{j\ge 1}$ of numbers ordered with multiplicity in a non-decreasing
order such that $ E_j \to \infty $. We denote the corresponding sequence $\{\phi_{j} \}_{j\ge 1}$ of (real-valued) eigenfunctions constituting an orthonormal basis
of the square-integrable functions $L^{2}(\M)$ on $\M$; the sequence $\{\phi_{j} \}_{j\ge 1}$ is uniquely determined up to
the spectral degeneracies (i.e., up to orthogonal transformations in each eigenspace of dimension $\ge 2$).

\subsection{Shnirelman's Theorem and Small-Scale Equidistribution}

Assuming w.l.o.g. that $\M$ is unit volume $\vol(\M)=1$, the celebrated Shnirelman's
Theorem ~\cite{Sn,Ze,CdV} asserts that if $\M$ is chaotic (i.e., the geodesic flow on $ \mathcal{M} $ is ergodic), then ``most" of the
$\{\phi_{j}\}$ are $L^{2}$-equidistributed. In particular, they are equidistributed in position space, i.e., there exists a density $1$ sequence $j_{k}$ such that for all
``nice" domains $\Ac\subseteq\M$ we have
\begin{equation}
\label{eq:L2 mass phi->A/M}
\lim\limits_{k\rightarrow\infty}\int\limits_{\Ac}\phi_{j_{k}}(x)^{2}dx= \vol(\Ac).
\end{equation}
Beyond Shnirelman's Theorem, Berry's universality conjecture ~\cite{Berry1,Berry2} implies that for a {\em generic} chaotic
manifold \eqref{eq:L2 mass phi->A/M} holds for $\Ac$ shrinking with $k$, slower than the Planck's scale $E_{j_{k}}^{-1/2}$.
More precisely, it states that there exists a density $1$ sequence $\{j_{k}\}_{k}$ so that if $r_{0}(E):\R_{> 0}\rightarrow\R_{> 0}$
satisfies $r_{0}(E)\cdot E^{1/2}\rightarrow\infty$ diverging arbitrarily slowly,
then, for $B_{x}(r)$ the radius $r$ geodesic ball in $\M$ centred at $x$, we have
\begin{equation}
\label{eq:|mass-exp|=o(r^d)}
\left|\int\limits_{B_{x}(r)}\phi_{j_{k}}(y)^{2}dy - \vol(B_{x}(r)) \right|
= o_{k\rightarrow\infty} (r^{d})
\end{equation}
uniformly for all $x\in\M$ and $r>r_{0}(E_{j_{k}})$, i.e.,
\begin{equation}
\label{eq:L2mass shrinking uniform}
\sup\limits_{\substack{x\in\M \\ r>r_{0}(E_{j_{k}})}}
\left|\frac{\int\limits_{B_{x}(r)}\phi_{j_{k}}(y)^{2}dy}{\vol(B_{x}(r))} - 1\right| \rightarrow 0.
\end{equation}


The following recent results are rigorous manifestations of the small-scale (``shrinking balls") statement \eqref{eq:L2mass shrinking uniform}.
Luo and Sarnak ~\cite[Theorem 1.2]{Luo-Sarnak} established the small-scale equidistribution
for Laplace eigenfunctions on the modular surface (assuming in addition that they are Hecke eigenfunctions)
where $r>E^{-\alpha}$ with a small $ \alpha>0 $, and Young \cite{Young}, conditionally on GRH, refined this estimate for $r>E^{-1/6+o(1)}$ holding for \emph{all} such eigenfunctions. Hezari and Rivi\`{e}re ~\cite{Hezari-Riviere1}, and independently Han ~\cite{Han1} established the equidistribution
for Laplace eigenfunctions on manifolds of negative curvature on logarithmic scale (i.e., $r>(\log{E})^{-\alpha}$, for some
$\alpha>0$), and Han ~\cite{Han2} considered random Laplace eigenfunctions on ``symmetric" manifolds, of high spectral degeneracy; here the higher the spectral degeneracy is the smaller the allowed scale is. More recently, Han and Tacy \cite{HanTacy} proved small-scale equidistribution for random Gaussian combinations of eigenfunctions on compact manifolds for $ r>E^{-1/2+o(1)} $, and de Courcy-Ireland \cite{DeCourcyIreland} showed that, with high probability, the $L^{2}$-mass of random Gaussian spherical harmonics is, up to a small error,
equidistributed, slightly above Planck scale.

\subsection{Toral Laplace eigenfunctions}

For the $d$-dimensional torus $\Tb^{d}=\R^{d}/\Z^{d}$, $d\ge 2$, there are high spectral degeneracies;
in this case Lester and Rudnick ~\cite[Theorem 1.1]{LeRu} proved that the small-scale equidistribution is satisfied by a generic Laplace eigenfunction (also considered by Hezari and Rivi\`{e}re ~\cite{Hezari-Riviere2}).
More precisely, they showed that every o.n.b. $\{\phi_{j}\}$ admits a density one subsequence $\{\phi_{j_{k}}\}$
of Laplace eigenfunctions obeying \eqref{eq:L2mass shrinking uniform}, with $r_{0}(E)=E^{-\alpha(d)}$, where $\alpha(d)$ is any number
smaller than
\begin{equation}
\label{eq:alpha(d) LeRu}
\alpha(d)<\frac{1}{2(d-1)},
\end{equation}
an (almost) optimal Planck-scale result for $d=2$, yet somewhat weaker than Berry's conjecture for $d> 2$.

One can express the real toral Laplace eigenfunctions explicitly as a sum of exponentials
\begin{equation}
\label{eq:fn sum exp}
f_{n}\left(x\right)=\sum_{\lambda\in\mathcal{E}_{n}}c_{\lambda}e\left(\left\langle x,\lambda\right\rangle \right), \hspace{10pt} (c_{-\lambda}=\overline{c_\lambda})
\end{equation}
for
\begin{equation}
\label{eq:S_d}
n\in S_{d}:=\{n=a_{1}^{2}+\ldots +a_{d}^{2}:\: a_{1},\ldots,a_{d}\in\Z \}
\end{equation}
expressible as a sum of $d$ integer squares,
and the corresponding frequencies $\lambda$ are the standard lattice points
\begin{equation}
\label{eq:E_n}
\Ec_{n} = \Ec_{d;n}=\{\lambda\in\Z^{d}:\: \|\lambda\|^{2}=n\}
\end{equation}
lying on the $(d-1)$-dimensional sphere (a circle for $ d=2 $) of radius-$\sqrt{n}$; in this case the energy is $E=E_{n}=4\pi^{2}n$.
We will assume w.l.o.g. that $f_{n}$ is $L^{2}$-normalised, equivalent to
\begin{equation}
\label{eq:BasicNormalization}
\|f_{n}\|_{L^{2}(\Tb^{d})}^{2} = \sum_{\lambda\in\mathcal{E}_{n}}\left|c_{\lambda}\right|^{2}=1.
\end{equation}

For every $ n\in S_d $, denote
\begin{equation}
\label{eq:N}
 N=N_{d;n}=\#\mathcal{E}_n.
 \end{equation}

When $d=2$, by Landau's theorem, $
\left\{ n\le x:\,n\in S_{2}\right\} \sim K\frac{x}{\sqrt{\log x}}
$
where $ K>0 $
is the ``Landau-Ramanujan constant". On average $N=N_{2;n}$ is of order of magnitude
$\sqrt{\log n}$; however, for a density one sequence
in $S_{2}$ we have
$
N=\left(\log n\right)^{\log2/2+o\left(1\right)}.
$
In general, for $n\in S_{2}$ we have \[N=n^{o\left(1\right)}.\]

For $d=3$, Siegel's theorem asserts that for $n\not\equiv0,4,7\,\left(8\right)$,
\[
N=N_{3;n}=n^{1/2+o\left(1\right)};
\]
since $x\mapsto2^{a}x$ is a bijection between the solutions to $x_{1}^{2}+x_{2}^{2}+x_{3}^{2}=n$
and $x_{1}^{2}+x_{2}^{2}+x_{3}^{2}=4^{a}n$, we can always assume
that $n\not\equiv0,4,7\,\left(8\right)$ with no loss of generality.

Granville and Wigman ~\cite[Theorem 1.2]{GranvilleWigman}
refined the aforementioned estimate by Lester-Rudnick for $d=2$. They proved that in this case, \eqref{eq:L2mass shrinking uniform}
is valid slightly above Planck-scale $r_{0}(E)=E^{-1/2+o(1)}$, for {\em all} eigenfunctions $f_{n}$ as in \eqref{eq:fn sum exp},
corresponding to numbers $n$ so that the lattice points $\Ec_{n}$ are well-separated (``Bourgain-Rudnick sequences"),
a condition satisfied ~\cite[Lemma 5]{Bourgain-Rudnick}
by ``generic" integers $n\in S_{2}$ in a strong quantitative sense, subsequently refined in ~\cite[Theorem 1.4]{GranvilleWigman},
see section \ref{sec:quasi corr}.

\subsection{Averaging mass w.r.t. ball centre}

For both the $2$-dimensional and the higher-dimensional tori it is possible to construct exceptional examples of sequences of
toral eigenfunctions where the equidistribution condition is not satisfied:
for $d\ge 2$ thin sequences ~\cite[Theorem $3.1$]{LeRu} $\{\phi_{j_{k}}\}$  of eigenfunctions violating condition \eqref{eq:|mass-exp|=o(r^d)} at Planck-scale
$r \cdot E_{j_{k}}^{1/2} \rightarrow \infty$, around the origin $x=0$, and even stronger, for $d\ge 3$
~\cite[Theorem $4.1$ (construction by J. Bourgain)]{LeRu} eigenfunctions violating \eqref{eq:|mass-exp|=o(r^d)} with $r \gg E^{-\alpha(d)}$ where $\alpha(d)> \frac{1}{2(d-1)}$,
again around the origin $x=0$.
In these cases, rather than keeping the ball centre $x=0$ at the origin, one may vary $x$,
and study whether the ``typical" discrepancy on the l.h.s. of \eqref{eq:|mass-exp|=o(r^d)} is {\em small}, even if the
existence of $x$ so that the l.h.s. of \eqref{eq:|mass-exp|=o(r^d)} is {\em not small} is known, so that, in particular,
\eqref{eq:L2mass shrinking uniform} is not satisfied.

A natural way to vary $x$ is to think of $x$ as {\em random}, drawn uniformly in $\Tb^{d}$.
We define the random variable
\begin{equation}
\label{eq:X_RV}
X_{f_{n},r}=X_{f_{n},r;x}:= \int\limits_{B_{x}(r)}f_{n}(y)^{2}dy,
\end{equation}
and are interested in the distribution of $X_{f_{n},r}$ where $x$ is drawn randomly uniformly in $\Tb^{d}$.
The relevant moments are: expectation
\begin{equation}
\label{eq:Expectation}
\E[X_{f_{n},r}] = \int\limits_{\Tb^{d}}X_{f_{n},r;x}dx,
\end{equation}
higher centred moments
\begin{equation}
\label{eq:centred moments}
\E[(X_{f_{n},r}-\E[X_{f_{n},r}])^{k}] = \int\limits_{\Tb^{d}}\left(X_{f_{n},r;x}-\E[X_{f_{n},r}]\right)^{k}dx, \hspace{1em}k\ge2,
\end{equation}
and in particular the variance
\begin{equation}
\label{eq:Variance}
\Vc(X_{f_{n},r}) = \E[(X_{f_{n},r}-\E[X_{f_{n},r}])^{2}].
\end{equation}

This approach of averaging the $L^{2}$-mass with respect to the ball centre (and keeping $f_{n}$ fixed) was pursued by Granville-Wigman ~\cite{GranvilleWigman} in the $2$-dimensional case, again slightly above the Planck scale $r>E^{-1/2+o(1)}$. In this regime, by proving an upper bound for $\Vc(X_{f_{n},r})$ beyond $(\E[X_{f_{n},r}])^{2} = O(r^4)$, valid for {\em all} $n\in S_{2}$, under some flatness assumption on $f_{n}$
(cf. Definition \ref{def:ultraflat} below), they established \eqref{eq:|mass-exp|=o(r^d)} for {\em ``typical"}, if {\em not all} $x\in\Tb^{2}$.
It would be desirable to find a regime where it is possible to analyse the precise asymptotic behaviour of the variance $\Vc(X_{f_{n},r})$
of $X_{f_{n},r}$, and, if possible, determine the limit distribution law for $X_{f_{n},r}$; our principal results below achieve both of these
in the $2$-dimensional case, and the former in the $3$-dimensional one (see theorems \ref{thm:VarMain} and \ref{thm:Var3D}).
Such an approach of bounding the discrepancy variance while averaging over ball centres was recently used by Sarnak ~\cite{Sa} for mass distribution of forms on symmetric spaces, and P. Humphries ~\cite{Humphries} for mass distribution of automorphic forms.

\subsection{Statement of the main results for $d=2,3$: asymptotics for the variance, CLT}

Our principal results below are applicable to ``flat" functions
for $d=2,3$, understood in suitable, more and less restrictive,
senses. For example, ``Bourgain's eigenfunction" \cite{Bourgain}
\begin{equation}
\label{eq:BourgainEF}
f_{n}\left(x\right)=\frac{1}{\sqrt{N}}\sum_{\lambda\in\mathcal{E}_{n}}\varepsilon_{\lambda}e\left(\left\langle x,\lambda\right\rangle \right)
\end{equation}
with $\varepsilon_{\lambda}=\pm1$ for every $\lambda\in\mathcal{E}_{n}$, i.e. corresponding to
$\left|c_{\lambda}\right|=N^{-1/2},$ satisfies any of the flatness conditions in the most restrictive sense. Denote $\Bc_{n}$ to be the class of Bourgain's eigenfunctions.

Our first principal result determines the precise asymptotic behaviour of the variance $\Vc(X_{f_n,r})$
for the $2$-dimensional case, and moreover asserts that the moments of the standardized random $L^{2}$-mass of $f_{n}$ are asymptotically Gaussian; we subsequently deduce a Central Limit Theorem (see Corollary \ref{cor:CLT_result}). For the sake of elegance of presentation, it is formulated for Bourgain's eigenfunctions \eqref{eq:BourgainEF}; below we formulate a more general result which holds for a larger class of flat eigenfunctions (see Theorem \ref{thm:VarMainGeneralized} in section \ref{sec:statement results strong}), and later
a result where the averaging over the ball centre $x$ is itself restricted to shrinking balls (Theorem \ref{thm:VarMainExplRestricted}
in section \ref{sec:RestrictedAverages}).

\begin{theorem}[Gaussian moments, $d=2$, Bourgain's eigenfunctions]
\label{thm:VarMain}
There exists a density one sequence
$S_{2}'\subseteq S_{2}$ so that the following holds. Let $r_{0}=r_{0}\left(n\right)=n^{-1/2}T_{0}\left(n\right)$ with
$T_{0}\left(n\right)\to\infty $.

\begin{enumerate}
\item Fix a number $\epsilon>0$, and suppose that $ T_0(n) < \left(\log n\right)^{\frac{1}{2}\log\frac{\pi}{2}-\epsilon} $.
Then as $n\to\infty$ along $S_{2}'$ we have
\begin{equation}
\label{eq:var asympt d=2 Bourgain}
\Vc\left(X_{f_{n},r}\right)\sim\frac{16}{3 \pi}r^{4}T^{-1}
\end{equation}
uniformly for all 
\begin{equation}
\label{eq:r0<r<n^-1/2*discr}
r_{0} < r <n^{-1/2}\left(\log n\right)^{\frac{1}{2}\log\frac{\pi}{2}-\epsilon}
\end{equation}
and $ f_n \in \Bc_{n} $, where $T:=n^{1/2}r.$

\item

Under the above notation let
\begin{equation}
\label{eq:standardizedX}
\hat{X}_{f_{n},r}:=\frac{X_{f_{n},r}-\E[X_{f_{n},r}]}{\sqrt{\Vc\left(X_{f_{n},r}\right)}}
\end{equation}
be the standardized random $L^{2}$-mass of $f_{n}$, $r_{1}=r_{1}(n)=n^{-1/2}T_{1}\left(n\right)$,
and suppose further that the sequence of numbers $T_{1}(n)>T_{0}(n)$ satisfies $T_{1}(n)=O\left(N^{\xi}\right)$
for every $\xi>0$.
Then for all $k\ge 3$ the $k$-th the moment of $\hat{X}_{f_{n},r}$ converges, for $n\rightarrow\infty$ along
$S_{2}'$, to the standard Gaussian moment
\begin{equation}
\label{eq:moments Gaussian lim}
\E[\hat{X}_{f_{n},r}^{k}] \rightarrow \E[Z^{k}],
\end{equation}
uniformly for $r_{0}<r<r_{1}$ and $ f_n \in \Bc_{n} $,
where $Z\sim N(0,1)$ is the standard Gaussian variable.

\end{enumerate}
\end{theorem}

The claimed uniform asymptotics \eqref{eq:var asympt d=2 Bourgain} of the variance means explicitly that, as $n\rightarrow\infty$ along $S_{2}'$,
one has
\begin{equation}
\label{eq:var unif asymp d=2}
\sup\limits_{\substack{r_0 < r < \left(\log n\right)^{\frac{1}{2}\log\frac{\pi}{2}-\epsilon} \\ f_n \in \Bc_{n}}} \left|\frac{\Vc\left(X_{f_{n},r}\right)}{\frac{16}{3 \pi}r^{4}T^{-1}}  -    1\right| \rightarrow 0
\end{equation}
and the uniform convergence \eqref{eq:moments Gaussian lim} of the moments means that for every $k\ge 3$,
\begin{equation*}
\sup\limits_{\substack{r_0 < r < r_1 \\ f_n \in \Bc_{n}}} \left|\E[\hat{X}_{f_{n},r}^{k}] - \E[Z^{k}]\right| \rightarrow 0.
\end{equation*}
Concerning the restricted range \eqref{eq:r0<r<n^-1/2*discr} in Theorem \ref{thm:VarMain} (and \eqref{eq:var unif asymp d=2}) 
for the possible radii, it is directly related to a well-known result on the angular distribution of
lattice points in $\Ec_{n}$, for generic $n\in S_{2}$. 
Namely, it was shown ~\cite{ErdosHall} that $\Ec_{n}$, projected by homothety to the unit circle, is equidistributed,
and moreover, a quantitative measure for the discrepancy is asserted (see section \ref{sec:ang distr} below,
and, in particular, \eqref{eq:D_n_epsilon}), satisfied by {\em generic} $n\in S_{2}$. 
Bourgain ~\cite{Bourgain} observed that $f_{n}\in \Bc_{n}$, when averaged over $x\in \Tb^{d}$, exhibits Gaussianity in the following sense. 
Let $T>0$ be a fixed number, and define the scaled function $\varphi_{x}:[-1,1]^{2}\rightarrow\R$ around $x$ as 
\begin{equation}
\label{eq:varphi rand x def}
\varphi_{x}(y):= f_{n}\left( x+ \frac{T}{\sqrt{n}}\cdot y\right),
\end{equation}
i.e. the trace of $f_{n}$ on the side-$2\frac{T}{\sqrt{n}}$ square centred at $x$. 
It was found ~\cite{Bourgain}, that, upon thinking of $x\in\Tb^{2}$ as {\em random}, and 
$\varphi_{x}(\cdot)$ as a {\em random field} indexed by $[-1,1]^{2}$, it converges, in a suitable sense, to
a particular {\em Gaussian} field (``monochromatic isotropic waves") on $\R^{2}$, restricted
to $[-1,1]^{2}$. This allows one to infer some results on the (deterministic) functions $f_{n}\in \Bc_{n}$
from the analogous results on the limit Gaussian random field. 
We may then reinterpret the quantitative version \eqref{eq:D_n_epsilon} of the angular equidistribution of lattice points 
as allowing the parameter $T$ in \eqref{eq:varphi rand x def} to grow as a (positive) logarithmic power of $n$,
while still retaining the said asymptotic Gaussianity, also allowing for the comparison between the mass distribution
of $f_{n}$ w.r.t. the position and mass distribution of monochromatic isotropic waves. 
Our intuition regarding the possibility of carrying on the explained ``de-randomisation" argument  
for establishing results of similar nature to 
the presented results was recently validated by Sartori ~\cite{Sartori}.

\vspace{2mm}

An application of the standard theory \cite[§XVI.3 Lemma 2]{Feller} allows us to infer a uniform Central Limit Theorem for the random variables
$\hat{X}_{f_{n},r}$ from the convergence \eqref{eq:moments Gaussian lim} of their respective moments to the Gaussian ones.

\begin{corollary}
\label{cor:CLT_result}
In the setting of Theorem \ref{thm:VarMain} part (2), the distribution of the random variables $\{\hat{X}_{f_{n},r}\}$ converges uniformly to
the standard Gaussian distribution: as $n\rightarrow\infty$ along $S_{2}'$
\begin{equation*}
\meas\{ x\in \mathbb{T}^2:\: \hat{X}_{f_{n},r;x} \le t\} \rightarrow
\frac{1}{\sqrt{2\pi}}\int\limits_{-\infty}^{t}e^{-z^{2}/2}dz,
\end{equation*}
uniformly for $t\in\R$, $r_{0}<r<r_{1}$ and $ f_n \in \Bc_{n} $.
\end{corollary}

For the $3$-dimensional case, for Bourgain's eigenfunctions, we only claim precise asymptotic result on $\text{\ensuremath{\mathcal{V}}}\left(X_{f_{n},r}\right)$, the good news being that
the claimed results are valid for {\em all} energies satisfying the natural congruence assumptions.

\begin{theorem}[Asymptotics for the variance for $d=3$, Bourgain's eigenfunctions]
\label{thm:Var3D}
There exists a number $\eta>0$ such that if $r_{0}=r_{0}(n)=n^{-1/2}T_{0}(n)$ with $T_{0}(n)\rightarrow\infty$,
then for all $n\not\equiv0,4,7\,\left(8\right)$ we have
\[
\text{\ensuremath{\mathcal{V}}}\left(X_{f_{n},r}\right)\sim r^{6}T^{-2},
\]
uniformly for $r_{0} <  r < n^{-1/2+\eta}$ and $ f_n \in \Bc_{n} $.
\end{theorem}

The meaning of the uniform statement in Theorem \ref{thm:Var3D} is that
\begin{equation}
\sup_{\begin{subarray}{c}
r_{0} < r < n^{-1/2+\eta} \\ f_n \in \Bc_{n} \end{subarray}}\left|\frac{\text{\ensuremath{\mathcal{V}}}\left(X_{f_{n},r}\right)}{r^{6}T^{-2}}-1\right|\to0\label{eq:AympVar3D}
\end{equation}
as $n\to\infty$ along $n\not\equiv0,4,7\,\left(8\right)$, cf. \eqref{eq:var unif asymp d=2} in the $2$-dimensional case.

\subsection{Statement of the main results for $d=2,3$: more general upper and lower bounds}

\label{sec:statement results weak}

Let $f_{n}$ be as in \eqref{eq:fn sum exp}, and consider the vector
\begin{equation}
\label{eq:v_def}
\underline{v}:=(|c_{\lambda}|^{2})_{\lambda \in\Ec_{n}} \in \R^{\mathcal{E}_n} 
\end{equation}
of the squared absolute values of its coefficients; we denote its normalised $\ell_{\infty}$-norm
\begin{equation}
\label{eq:vnorm inf}
[\underline{v}]_{\infty} := N \cdot \max\limits_{\lambda\in\Ec_{n}}|c_{\lambda}|^{2}.
\end{equation}

\begin{definition}[Ultraflat functions {~\cite[Definition 1.9]{GranvilleWigman}}]
\label{def:ultraflat}

We say that an eigenfunction $f_{n}$ in \eqref{eq:fn sum exp} is $\epsilon$-ultraflat if its coefficients satisfy
\begin{equation}
\label{eq:ultra_flat_cond}
[\underline{v}]_{\infty} \le N^{\epsilon}.
\end{equation}
Denote  $\Uc_{n;\epsilon}$ to be the class of $\epsilon$-ultraflat functions.

\end{definition}

The following couple of theorems establish more general upper and lower bounds on $\Vc(X_{f_n,r})$
in the $2$ and $3$-dimensional cases respectively.

\begin{theorem}[Bounds for the variance for ultra-flat eigenfunctions, $d=2$]
\label{thm:UpperBound2d}
There exists a density $1$ sequence $S_{2}'\subseteq S_{2}$ and an absolute constant $C>0$ such that
for every $ \epsilon>0 $, $ \eta >0 $, $r_{0}=r_{0}(n)=n^{-1/2}T_{0}(n)$ with
$T_{0}(n)\rightarrow\infty$ arbitrarily slowly, and $r=n^{-1/2}T>r_{0}$, as $n\to\infty$ along $S_{2}'$ we have
\begin{equation}
\label{eq:bounds var ultraflat d=2}
T^{-1}N^{-2\epsilon}\ll    \frac{\Vc(X_{f_n,r})}{r^{4}} \ll N^{\epsilon}\cdot \left(T^{-1}+(\log{n})^{-\frac{1}{2}\log{\frac{\pi}{2}}+\eta}   \right)
\end{equation}
uniformly for $r_{0} < r< Cn^{-1/2}N^{1-\epsilon}$ and $f_{n}\in \Uc_{n;\epsilon}$,
with the constant involved in the ``$\ll$"-notation in \eqref{eq:bounds var ultraflat d=2}  is absolute for the lower bound, and depends only on $\eta$ for the upper bound.
Moreover, the upper bound is valid for the extended range $r > r_{0}$ (with no upper bound on $r$ imposed), and the lower bound is valid for every $ n\in S_2 $.
\end{theorem}

\begin{theorem}[Bounds for the variance for ultra-flat eigenfunctions, $d=3$]
\label{thm:UpperBound3d}
There exists a number $\eta>0$ and a constant $C>0$ such that for every $\epsilon>0$,
$r_{0}=r_{0}(n)=n^{-1/2}T_{0}(n)$ with $T_{0}(n)\rightarrow\infty$ arbitrarily slowly,
$r=n^{-1/2}T>r_{0}$, and $n\not\equiv 0,4,7\,\left(8\right)$ we have
\begin{equation}
\label{eq:bounds var ultraflat d=3}
T^{-2}N^{-2\epsilon}  \ll \frac{\Vc(X_{f_{n},r})}{r^{6}} \ll N^{\epsilon} \left( T^{-2}+n^{-\eta}\right),
\end{equation}
uniformly
for $r_{0} < r < Cn^{-1/2}N^{1-\epsilon}$ and $f_{n}\in \Uc_{n;\epsilon}$, where the constants involved in the ``$\ll$"-notation
are absolute. Moreover, the upper bound in \eqref{eq:bounds var ultraflat d=3} is valid for the extended range $r > r_{0}$.

\end{theorem}

For Bourgain's eigenfunctions, the proofs of Theorem \ref{thm:UpperBound2d} and Theorem \ref{thm:UpperBound3d} yield slightly stronger bounds compared to \eqref{eq:bounds var ultraflat d=2} and \eqref{eq:bounds var ultraflat d=3}, namely
\begin{equation*}
T^{-1}\ll    \frac{\Vc(X_{f_n,r})}{r^{4}} \ll  T^{-1}+(\log{n})^{-\frac{1}{2}\log{\frac{\pi}{2}}+\epsilon}
\end{equation*} for $ d=2 $, and
\begin{equation*}
T^{-2}  \ll \frac{\Vc(X_{f_{n},r})}{r^{6}} \ll   T^{-2}+n^{-\eta}
\end{equation*} for $ d=3 $.

\subsection{Outline of the paper}

The rest of the paper is organised as follows.
In section \ref{sec:statement results strong} we formulate Theorem \ref{thm:VarMainGeneralized}, which, on one hand generalizes Theorem \ref{thm:VarMain} for a larger class of flat eigenfunctions, and on the other hand, explicates
a sufficient condition on $ n\in S_2 $ for its statements to hold; a few examples of application of
Theorem \ref{thm:VarMainGeneralized}, corresponding to different asymptotic behaviour of the
variance \eqref{eq:var asympt d=2 precise}, are also discussed.
Section \ref{sec:Proof_Main_thm_part1} is dedicated to giving a proof of
the first part of Theorem \ref{thm:VarMain} (resp. $1$st part of Theorem \ref{thm:VarMainGeneralized}), whereas the second part of Theorem \ref{thm:VarMain} (resp. $2$nd part of Theorem \ref{thm:VarMainGeneralized}) is proved in section \ref{sec:Proof_Main_thm_part2}.
Theorem \ref{thm:Var3D}, claiming the precise asymptotics for the $L^{2}$-mass variance for Bourgain's eigenfunctions
in $3$d, is proved in section \ref{sec:Proof_3d_theorem}.

In section \ref{sec:ProofOfBoundsThm} we
prove the various upper and lower bounds asserted by theorems \ref{thm:UpperBound2d} and \ref{thm:UpperBound3d}.
A refinement of Theorem \ref{thm:VarMainGeneralized}, where rather than draw $x$ w.r.t. the uniform measure on the full torus,
$x$ is drawn on balls slightly above Planck scale, is presented in section \ref{sec:RestrictedAverages}, and the additional subtleties
of its proof as compared to the proof of Theorem \ref{thm:VarMainGeneralized} are highlighted.
Finally, section \ref{sec:AuxLemmasProof} contains the proofs of all auxiliary lemmas, postponed in course of the proofs of
the various results.

\subsection*{Acknowledgements}

The authors of this manuscript wish to express their gratitude to J. Benatar, A. Granville, P. Kurlberg, Z. Rudnick,
P. Sarnak and M. Sodin for numerous stimulating and fruitful discussions concerning various aspects of our work, and
their interest in our research. It is a pleasure to thank the anonymous referee for his comments on
an earlier version of this manuscript. The research leading to these results has received funding from the European Research Council under the European Union's Seventh Framework Programme (FP7/2007-2013), ERC grant agreement n$^{\text{o}}$ 335141.

\section{On Theorem \ref{thm:VarMain}: CLT for mass distribution, $d=2$}

\label{sec:statement results strong}

In this section we focus on Theorem \ref{thm:VarMain}.
Our first goal is to formulate a result, that on one hand generalises the statement of Theorem \ref{thm:VarMain}
to a larger class of eigenfunctions, and, on the other hand, provides a more explicit control over the
generic numbers $n\in S_{2}$. To this end we discuss the angular distribution of $\lambda\in\Ec_{n}$
(section \ref{sec:ang distr}), and the spectral correlations (section \ref{sec:quasi corr}), also used in the course of the proof of the $3$-dimensional Theorem \ref{thm:Var3D}; we will be able to formulate Theorem \ref{thm:VarMainGeneralized}, as prescribed above,
by appealing to these. In section \ref{eq:examples varying theta} we consider a few scenarios when Theorem \ref{thm:VarMainGeneralized} is applicable, prescribing different asymptotic behaviour for the variance \eqref{eq:var asympt d=2 precise}.

\subsection{Angular equidistribution of lattice points}
\label{sec:ang distr}

For every $\lambda=\left(\lambda_{1},\lambda_{2}\right)\in\mathcal{E}_{n}$,
write $\lambda_{1}+i\lambda_{2}=\sqrt{n}e^{i\phi}$, and denote the
various angles by
\[
0\le\phi_{1}<\phi_{2}<\dots<\phi_{N}<2\pi.
\]
Recall that the discrepancy of the sequence $\phi_{j}$ is defined
by
\begin{equation}
\label{eq:Discrepancy_2d}
\Delta\left(n\right)=\sup_{0\le a\le b\le2\pi}\left|\frac{1}{N}\cdot \#\left\{ 1\le j\le N:\,\phi_{j}\in\left[a,b\right]\,\text{mod\,}2\pi\right\} -\frac{\left(b-a\right)}{2\pi}\right|.
\end{equation}

For every $ \epsilon>0 $, we say that $ n\in S_2 $ satisfies the hypothesis $ \mathcal{D}(n,\epsilon) $ if \begin{equation}
\label{eq:D_n_epsilon}
\Delta\left(n\right)\le \left(\log n\right)^{-\frac{1}{2}\log\frac{\pi}{2}+\epsilon}.
\end{equation}
By Erd\H{o}s-Hall \cite[Theorem 1]{ErdosHall}, there exists a density
one sequence $S_2'(\epsilon)\subseteq S_2 $ such that $ \mathcal{D}(n,\epsilon) $ is satisfied for every $n\in S_2'(\epsilon) $. By a standard diagonalization argument, there exists a density one sequence $ S_2'\subseteq S_2 $ such that $\mathcal{D}(n,\epsilon) $ is satisfied  for \emph{every} $ \epsilon>0 $ and $n\in S_2' $ sufficiently large.
In particular, the angles $\left\{ \phi_{j}\right\} $ are equidistributed
mod $2\pi$ along this sequence, i.e., the lattice points are equidistributed
on the corresponding circles.

\subsection{Spectral correlations in $2d$ (and $3d$)}
\label{sec:quasi corr}

For $d=2$, while computing the moments of $X_{f_{n},r}$ (e.g. for Bourgain's eigenfunction \eqref{eq:BourgainEF}),
with $x$ drawn uniformly on the whole of $\Tb^{2}$, one exploits the orthogonality relations
\begin{equation*}
\int\limits_{\Tb^{2}}e(\langle  \lambda , x  \rangle)dx = \begin{cases}
0 &\lambda\ne 0 \\ 1 &\lambda=0
\end{cases}
\end{equation*}
for $\lambda\in\Z^{2}$ to naturally encounter the length-$l$ spectral correlation problem.
That is, for $l\ge 2$ and $n\in S_{2} $ one is interested in the size of the length-$ l $ spectral correlation set
\begin{equation}
\label{eq:Sc correlations def}
\Sc_{n}(l) = \left\{ (\lambda^{1},\ldots,\lambda^{l})\in(\Ec_{n})^{l}:\: \sum\limits_{i=1}^{l}\lambda^{i}=0  \right\},
\end{equation}
which, by an elementary congruence obstruction argument modulo $ 2 $, is only non-empty for $l=2k$ even.

In this case $l=2k$ we further define the {\em diagonal} correlations set to be
all the permutations of tuples of the form $(\lambda^{1},-\lambda^{1},\ldots, \lambda^{k},-\lambda^{k})$:
\begin{equation}
\label{eq:Dc diag def}
\Dc_{n}(l) = \left\{ \pi(\lambda^{1},-\lambda^{1},\ldots,\lambda^{k},-\lambda^{k}): \lambda^{1},\ldots,\lambda^{k}\in (\Ec_{n})^{k},\,\pi\in S_{l}\right\}.
\end{equation}
The set $\Dc_{n}$ is dominated by non-degenerate tuples (i.e. $\lambda^{i}\ne \pm\lambda^{j}$ for $ i\ne j $), hence its size is asymptotic to
$$|\Dc_{n}(l)|=  \frac{(2k)!}{2^{k}\cdot k!}N^{k}\cdot \left(1 + O_{N\rightarrow\infty}\left( \frac{1}{N} \right)  \right).$$

Clearly, $\Dc_{n}(l)\subseteq \Sc_{n}(l)$ so that, in particular $\Sc_{n}(l) \gg N^{l/2}$. To the other end,
we have $\Sc_{n}(2) = \Dc_{n}(2)$ by the definition,
and both the precise statement
\begin{equation}
\label{eq:Zygmund 4-corr}
\Sc_{n}(4) = \Dc_{n}(4)
\end{equation}
(used for the variance computation below) and the bound
\begin{equation*}
|\Sc_{n}(l)| = O_{N\rightarrow\infty}(N^{l-2})
\end{equation*}
follow from Zygmund's elementary observation ~\cite{Zygmund}. For $ l=6 $, Bourgain (published in ~\cite{K-K-W}) improved Zygmund's bound to
\begin{equation*}
|\Sc_{n}(6)| = o_{N\rightarrow\infty}(N^{4});
\end{equation*}
this was improved ~\cite{BombieriBourgain} to
\begin{equation*}
|\Sc_{n}(6)| = O_{N\rightarrow\infty}(N^{7/2}),
\end{equation*}
valid  for {\em all} $n\in S_{2}$.

If one is willing to excise a thin sequence in $S_{2}$, then the more striking
estimate \cite{BombieriBourgain}
\begin{equation*}
|\Sc_{n}(6)| = |\Dc_{n}(6)| + O(N^{3-\gamma}),
\end{equation*}
with some $\gamma >0$,
is valid for a density $1$ sequence $S_{2}'\subseteq S_{2}$.
More generally \cite{Bourgain}, for every $l\ge 6$ even, there exists
a density $1$ sequence $S_{2}'(l)\subseteq S_{2}$ and a number $\gamma_{l}>0$ such that
\begin{equation}
\label{eq:corr diag dom l}
|\Sc_{n}(l)| = |\Dc_{n}(l)| + O(N^{l/2-\gamma_l})
\end{equation}
along $n\in S_{2}'(l)$.
A standard diagonal argument then yields the
existence of a density $1$ sequence $S_2'\subseteq S_{2}$ so that \eqref{eq:corr diag dom l} is valid for {\em all} even $l\ge 4$.

\begin{definition}
Given an even number $l=2k\ge 2$ we say that a sequence $S_{2}'\subseteq S_{2}$ satisfies the length-$l$
{\bf diagonal domination} assumption if there exists a number $\gamma=\gamma_l >0$ so that \eqref{eq:corr diag dom l} holds.
\end{definition}

\vspace{2mm}

For the $3$-dimensional case under the consideration of Theorem \ref{thm:Var3D} the analogous estimates to \eqref{eq:corr diag dom l} are required
to evaluate the relevant moments \eqref{eq:centred moments} of
$X_{f_{n},r}$. We define $\Sc_{3;n}$ and $\Dc_{3;n}$ analogously to
\eqref{eq:Sc correlations def} and \eqref{eq:Dc diag def} respectively, this time the $\lambda^{i}$ are lying on the $2$-sphere of radius $\sqrt{n}$.
Unlike the lattice points lying on circles, Zygmund's argument is not applicable for the $2$-sphere, so that an analogue of
\eqref{eq:Zygmund 4-corr} is not valid; luckily the asymptotic statement
\begin{equation}
\label{eq:3d 4-corr}
|\Sc_{3;n}(4)| = |\Dc_{3;n}(4)| + O\left( N^{7/4 +\epsilon}  \right),
\end{equation}
a key input to the variance computation in Theorem \ref{thm:Var3D},
was recently established ~\cite{BenatarMaffucci}. It was also shown in \cite{BenatarMaffucci} that the asymptotic diagonal domination for the higher length correlations sets does not hold in the $ 3 $-dimensional case.

\subsection{A more general version of Theorem \ref{thm:VarMain}, with explicit control over $S_{2}'$}
\label{sec:VarMainExpl}

\vspace{2mm}

We are interested in extending Theorem \ref{thm:VarMain} to a larger class of eigenfunctions. To this end, we introduce the following notation:

\begin{notation}
	
	Let $f_{n}$ be an eigenfunction on the $2$-torus corresponding to coefficients $(c_{\lambda})_{\lambda\in\Ec_{n}}$ via
	\eqref{eq:fn sum exp}, and $\underline{v}\in\R^{\mathcal{E}_n}\simeq \R^{N}$ as above.
	
	\begin{enumerate}
		
		\item
		
		Denote
		\begin{equation}
		\label{eq:A4Def}
		A_{4} = A_{4}(\underline{v}) = N\sum\limits_{\lambda\in\Ec_{n}} |c_{\lambda}|^{4} = N\cdot \|\underline{v}\|^{2}.
		\end{equation}
		
		\item Given $\lambda\in\Ec_{n}$ let $\lambda_{+}$ be the clockwise nearest neighbour of $\lambda$
		on $\sqrt{n}\mathcal{S}^{1}$, and
		\begin{equation}
		\label{eq:BasicVariablesV}
		V\left(\underline{v} \right):=
		N\sum_{\lambda\in\mathcal{E}_{n}}\left|\left|c_{\lambda_{+}}\right|^{2}-\left|c_{\lambda}\right|^{2}\right|.
		\end{equation}
		
		\item
		Let
		\begin{equation}
		\label{eq:alphaDef}
		\widetilde{V}(\underline{v}) = \frac{[\underline{v}]_{\infty} \cdot V(\underline{v})}{A_{4}(\underline{v})}.
		\end{equation}
		
	\end{enumerate}
	
\end{notation}

The following Lemma, proved in section \ref{sec:AuxLemmasProof},
summarizes some basic properties of the quantities in (\ref{eq:vnorm inf}),
(\ref{eq:A4Def}), (\ref{eq:BasicVariablesV}) and (\ref{eq:alphaDef}):
\begin{lemma}
	\label{lem:BasicVarProp}
	We have
	\begin{enumerate}
		\item $1\le A_{4} \le[\underline{v}]_{\infty}.$
		\item $[\underline{v}]_{\infty} \le1+V\left(\underline{v} \right)$.
		\item $V\left(\underline{v} \right) \le\widetilde{V}(\underline{v})\le V\left(\underline{v} \right)\left(1+V\left(\underline{v} \right)\right)$.
	\end{enumerate}
\end{lemma}

By \eqref{eq:BasicNormalization} we have that
\begin{equation}
\label{eq:A4<->theta}
A_{4} = \cos(\theta)^{-2},
\end{equation}
where $\theta = \theta_{f_n} = \theta(\underline{v},\underline{v_{0}})$ is the angle between $\underline{v}$ and the vector
$\underline{v_{0}}= (\frac{1}{N})_{\lambda\in\Ec_{n}}$ corresponding to Bourgain's eigenfunctions,
hence $ \theta $ reflects the proximity of
$f_{n}$ to Bourgain's eigenfunction; by the first part of Lemma \ref{lem:BasicVarProp}, the angle $\theta$ is restricted to the interval $ \left[0,\arccos \left(1/\sqrt{N}\right)\right] \subseteq [0,\pi/2) $.

\begin{definition}[Classes $\Fc_{1}(n;T(n),\eta(n))$ and $\Fc_{2}(n;T(n),\eta(n))$, $ d=2 $]
	Given a sequence $T(n)\rightarrow\infty$ and a sequence $\eta(n)>0$ we define:
	
	\begin{enumerate}
		
		\item A sequence
		$\{\Fc_{1}(n;T(n),\eta(n))\}_{n}$ of families of functions consisting for  $n\in S_2$ of all functions $f_{n}$ as in \eqref{eq:fn sum exp} satisfying
		\begin{equation}\label{eq:F_1_def}\Fc_{1}(n;T(n),\eta(n)) = \left\{f_{n}:\: \widetilde{V}(\underline{v})< \eta(n) \cdot \frac{T(n)}{\log T(n)} \right\}.\end{equation}
		
		\item
		
		A sequence $\{\Fc_{2}(n;T(n),\eta(n))\}_{n}$ of families of functions
		consisting for $n\in S_2$ of all functions $f_{n}$ as in \eqref{eq:fn sum exp} satisfying
		\begin{equation}\label{eq:F_2_def}\Fc_{2}(n;T(n),\eta(n)) = \left\{f_{n}:\: [\underline{v}]_{\infty} < T(n)^{\eta(n)} \right\},\end{equation}
		where we recall the notation \eqref{eq:vnorm inf} for $[\underline{v}]_{\infty}$.
		
	\end{enumerate}
	
\end{definition}

We are now in a position to state the generalized version of Theorem \ref{thm:VarMain}:

\begin{theorem}
\label{thm:VarMainGeneralized}
Let $r_{0}=r_{0}\left(n\right)=n^{-1/2}T_{0}\left(n\right)$ with
$T_{0}\left(n\right)\to\infty$, and $\eta(n)>0$ any vanishing sequence $\eta(n)\rightarrow 0$.

\begin{enumerate}
\item Fix a number $\epsilon>0$, and suppose that $ T_0(n) < \left(\log n\right)^{\frac{1}{2}\log\frac{\pi}{2}-\epsilon} $.
Then, if $S_{2}'\subseteq S_{2}$ is a sequence satisfying $ \mathcal{D}(n,\epsilon/2)$ for all $n\in S_{2}'$,
as $n\rightarrow\infty$ along $S_{2}'$, we have
\begin{equation}
\label{eq:var asympt d=2 precise}
\Vc\left(X_{f_{n},r}\right)\sim\frac{16}{3 \pi \cos^{2}\theta_{f_{n}}}r^{4}T^{-1}
\end{equation}
with $\theta_{f_{n}}$ as in \eqref{eq:A4<->theta},
uniformly for all $r_{0} < r <n^{-1/2}\left(\log n\right)^{\frac{1}{2}\log\frac{\pi}{2}-\epsilon}$ and $f_{n}\in\Fc_{1}(n;T(n),\eta(n))$,
where $T:=T(n)=n^{1/2}r.$

\item
		
Let $k\ge 3$ be an integer, $r_{1}=r_{1}(n)=n^{-1/2}T_{1}\left(n\right)$, and suppose further that the sequence of numbers $T_{1}(n)>T_{0}(n)$ satisfies $T_{1}(n)=O\left(N^{\xi}\right)$ for every $\xi>0$. Suppose that $S_{2}'\subseteq S_{2}$ is a sequence satisfying the length-$2k$ diagonal domination assumption and the hypothesis $\mathcal{D}(n,\epsilon)$ for all $n\in S_{2}'$. Then the $k$-th the moment of $\hat{X}_{f_{n},r}$ converges, as $n\rightarrow\infty$ along $S_{2}'$, to the standard Gaussian moment
\begin{equation*}
\E[\hat{X}_{f_{n},r}^{k}] \rightarrow \E[Z^{k}],
\end{equation*}
uniformly for $r_{0}<r<r_{1}$ and $f_{n}\in \Fc_{2}(n;T(n),\eta(n))$
where $Z\sim N(0,1)$ is the standard Gaussian variable.		
\end{enumerate}
\end{theorem}

Section \ref{eq:examples varying theta} exhibits a few scenarios when Theorem \ref{thm:VarMainGeneralized} is applicable;
as in these the true asymptotic behaviour of the variance \eqref{eq:var asympt d=2 precise} genuinely varies together with $\theta_{f_{n}}$,
this demonstrates that $\theta_{f_{n}}$ (and hence $A_{4}$) is the proper flatness measure of $f_{n}$, see also examples \ref{ex:Bourgain} and \ref{ex:flat vs nonflat}.

\vspace{2mm}


\begin{corollary}

In the setting of Theorem \ref{thm:VarMainGeneralized} part (2), the distribution of the random variables $\{\hat{X}_{f_{n},r}\}$ converges uniformly to the standard Gaussian distribution: as $n\rightarrow\infty$ along $S_{2}'$
	\begin{equation*}
	\meas\{ x\in\mathbb{T}^2 :\: \hat{X}_{f_{n},r;x} \le t\} \rightarrow
	\frac{1}{\sqrt{2\pi}}\int\limits_{-\infty}^{t}e^{-z^{2}/2}dz,
	\end{equation*}
	uniformly for $t\in\R$, $r_{0}<r<r_{1}$, and
	$f_{n}\in \Fc_{2}(n;T(n),\eta(n))$.
\end{corollary}


\subsection{Some examples of application of Theorem \ref{thm:VarMainGeneralized}}
\label{eq:examples varying theta}

\begin{example}
\label{ex:Bourgain}
Let $ f_n $ be Bourgain's eigenfunction, so that $ [\underline{v}]_{\infty} =A_4 =1 $ and $ V\left(\underline{v} \right) = \widetilde{V}(\underline{v}) =0 $. For every $ \eta(n)>0,T(n)>1 $ we have
\[ \Bc_{n} \subseteq \Fc_{1}(n;T(n),\eta(n)) \cap \Fc_{2}(n;T(n),\eta(n)). \] Hence Theorem \ref{thm:VarMainGeneralized} implies Theorem \ref{thm:VarMain}.
\end{example}

The following example exhibits a scenario  when an application of Theorem \ref{thm:VarMainGeneralized} yields
a Central Limit Theorem for $X_{f_{n},r}$, corresponding to asymptotic behaviour of the respective variance $\Vc(X_{f_{n},r})$ which is very different from the behaviour in Theorem \ref{thm:VarMain}.

\begin{example}
\label{ex:flat vs nonflat}

Let $\epsilon>0$, $r_{0}$, and $ T_0(n) $ as in Theorem \ref{thm:VarMainGeneralized}, and $r_{1}= r_1(n)=n^{-1/2} T_{1}(n) > r_{0}$
with $T_{1}(n) \le (\log{n})^{\frac{1}{2}\log{\frac{\pi}{2}}-\epsilon}$. There exists
a density $1$ sequence $S_{2}'\subseteq S_{2}$ so that the following holds.
Let $t=t(n)\in (0,1)$ be a number satisfying $t(n) \gg \frac{1}{T_{0}(n)^{\xi}}$ for every $\xi>0$,
such that $N\cdot t$ is an integer. We choose an ordering $\lambda^{1},\lambda^{2},\ldots \lambda^{N}\in\Ec_{n}$
such that for every $1\le i \le N-1$ we have that $\lambda^{i+1}$ is the (clockwise) nearest neighbour
$\lambda^{i+1}=\lambda^{i}_{+}$, and set
\[
\left(\left|c_{\lambda^{1}}\right|^{2},,\dots,\left|c_{\lambda^{N}}\right|^{2}\right)=(\underset{\begin{subarray}{c}
Nt\end{subarray}\text{\,times}}{\underbrace{\left(Nt\right)^{-1},\dots\dots,\left(Nt\right)^{-1}}},0\dots,0).
\]
Then
\begin{equation}
\label{eq:var asympt nonflat}
\Vc(X_{f_{n},r}) \sim \frac{16}{3 \pi}r^{4}t^{-1}T^{-1},
\end{equation}
uniformly for $r_{0}<r=n^{-1/2}T<r_{1}$, and $f_{n}$ with coefficients $c_{\lambda}$ as above.
If, in addition, we have $T_{1}(n)=O(N^{\xi})$ for every $\xi>0$, then the distribution of the standardised random variable
$\hat{X}_{f_{n},r}$ converges to standard Gaussian uniformly.

\end{example}

Comparing \eqref{eq:var asympt d=2 Bourgain} to \eqref{eq:var asympt nonflat}
we observe that the asymptotic behaviour of the variance for the flat and the non-flat functions
respectively is genuinely different, provided that we choose $t(n)\rightarrow 0$;
we infer that the proposed flatness measure is the natural choice for this problem.
One can also generalise Theorem \ref{thm:VarMain} as follows:

\begin{corollary}
\label{cor:VarAsympGen}

Let $\epsilon$, $r_{0}$, $T_{0}(n)$, $r_{1}$ and $T_{1}(n)$ be as in Theorem \ref{thm:VarMainGeneralized},
and $g:\Sc^{1}\rightarrow\R$ a non-negative function of bounded variation such that $ \|g\|_{L^{1}(\Sc^{1})}=1$. For $n\in S_{2}$ and $\lambda\in \Ec_{n}$ we set
$|\widetilde{c_{\lambda}}|^{2} := g(\lambda/\sqrt{n})$,
and normalise the vector $\widetilde{\underline{v}}:=(|\widetilde{c_{\lambda}}|^{2})_{\lambda\in\Ec_{n}}$ by
setting $\underline{v} := \frac{\widetilde{\underline{v}}}{\|\widetilde{\underline{v}}\|_{1}}$, i.e.
\begin{equation}
\label{eq:v BV norm}
v:=(|{c_{\lambda}}|^{2})_{\lambda\in\Ec_{n}} =
\left(\frac{|\widetilde{c_{\lambda}}|^{2}}{\sum\limits_{\mu\in\Ec_{n}}|\widetilde{c_{\mu}}|^{2}}\right)_{\lambda\in\Ec_{n}}.
\end{equation}

Then along a generic sequence $S_{2}'\subseteq S_{2}$ we have
\begin{equation*}
\Vc(X_{f_{n},r}) \sim \frac{16}{3 \pi}\|g\|_{L^{2}(\Sc^{1})}^{2}r^{4}T^{-1},
\end{equation*}
uniformly for
$r_{0}<r=n^{-1/2}T<r_{1}$, and $f_{n}$ with coefficients $c_{\lambda}$ as in \eqref{eq:v BV norm}.
If, in addition, we have $T_{1}(n)=O(N^{\xi})$ for every $\xi>0$, then the distribution of the standardised random variable
$\hat{X}_{f_{n},r}$ converges to standard Gaussian.
\end{corollary}

\begin{proof}
By Koksma's inequality (see e.g. \cite{KuipersNiederreiter}),
$A_{4}\left(\underline{v} \right)\sim\left\Vert g\right\Vert _{2}^{2}$
along a density one sequence in $S_{2}$. Also note that $$V(\underline{v}) \ll V\left(g\right),$$
with the l.h.s. as in \eqref{eq:BasicVariablesV}, and r.h.s. the variation of $g$ on $\Sc^{1}$.
In light of Lemma \ref{lem:BasicVarProp}, both parts of Corollary \ref{cor:VarAsympGen} follow from Theorem \ref{thm:VarMainGeneralized}.
\end{proof}

\section{Notation}
For the convenience of the reader, we summarize here the notation used in our paper.\linebreak

$ S_d=\{n=a_{1}^{2}+\ldots +a_{d}^{2}:\: a_{1},\ldots,a_{d}\in\Z \} $: the set of integers expressible as a sum of $ d $ squares, see \eqref{eq:S_d}.\\

$\Ec_{n} = \Ec_{d;n}=\{\lambda\in\Z^{d}:\: \|\lambda\|^{2}=n\}$:  the standard lattice points lying on the $(d-1)$-dimensional sphere (a circle for $ d=2 $) of radius-$\sqrt{n}$, see \eqref{eq:E_n}.\\

$f_{n}\left(x\right)=\sum\limits_{\lambda\in\mathcal{E}_{n}}c_{\lambda}e\left(\left\langle x,\lambda\right\rangle \right)$: the toral Laplace eigenfunctions, see \eqref{eq:fn sum exp}.\\

$ N=N_{d;n}=\#\mathcal{E}_n$: the number of lattice points on the $(d-1)$-dimensional sphere (a circle for $ d=2 $) of radius-$\sqrt{n}$, see \eqref{eq:N}.\\

$B_{x}(r)$: the radius $r$ geodesic ball in $\mathbb{T}^d$ centred at $x$.\\

$X_{f_{n},r}=X_{f_{n},r;x}= \int\limits_{B_{x}(r)}f_{n}(y)^{2}dy$: the $L^{2}$-mass of $ f_n $ restricted to $ B_{x}(r) $, where $ x $ is drawn randomly uniformly in $ \mathbb{T}^d $, see \eqref{eq:X_RV}.\\

$\E[X_{f_{n},r}] = \int\limits_{\Tb^{d}}X_{f_{n},r;x}dx$: the expected value of $X_{f_{n},r} $, see \eqref{eq:Expectation}.\\

$\Vc(X_{f_{n},r}) = \E[(X_{f_{n},r}-\E[X_{f_{n},r}])^{2}]$: the variance of  $X_{f_{n},r} $, see \eqref{eq:Variance}.\\

$\hat{X}_{f_{n},r}:=\frac{X_{f_{n},r}-\E[X_{f_{n},r}]}{\sqrt{\Vc\left(X_{f_{n},r}\right)}}$: the standardized random  $L^{2}$-mass of $ f_n $, see \eqref{eq:standardizedX}.\\

$ T = n^{1/2}r $.\\

$\underline{v}=(|c_{\lambda}|^{2})_{\lambda \in\Ec_{n}} \in \R^{\mathcal{E}_n} $:
the vector of the squared absolute values of the coefficients of $f_n $, see \eqref{eq:v_def}.\\

$[\underline{v}]_{\infty} = N \cdot \max\limits_{\lambda\in\Ec_{n}}|c_{\lambda}|^{2}$: the normalised $\ell_{\infty}$-norm of $ \underline{v} $, see \eqref{eq:vnorm inf}.\\

$\mathcal{B}_n$: the class of Bourgain's eigenfunctions $f_{n}\left(x\right)=\frac{1}{\sqrt{N}}\sum\limits_{\lambda\in\mathcal{E}_{n}}\varepsilon_{\lambda}e\left(\left\langle x,\lambda\right\rangle \right)$,
where $\varepsilon_{\lambda}=\pm1$ for every $\lambda\in\mathcal{E}_{n}$, see \eqref{eq:BourgainEF}.\\

$\Uc_{n;\epsilon}$: the class of $ \epsilon $-ultraflat functions, where $ [\underline{v}]_{\infty} \le N^{\epsilon} $, see \eqref{eq:ultra_flat_cond}.\\

$A_{4} = A_{4}(\underline{v}) = N\sum\limits_{\lambda\in\Ec_{n}} |c_{\lambda}|^{4} = N\cdot \|\underline{v}\|^{2}$, see \eqref{eq:A4Def}.\\

$\theta = \theta_{f_n} = \theta(\underline{v},\underline{v_{0}})$: the angle between $\underline{v}$ and the vector
$\underline{v_{0}}= (\frac{1}{N})_{\lambda\in\Ec_{n}}$ corresponding to Bourgain's eigenfunctions, see \eqref{eq:A4<->theta}.\\

$V\left(\underline{v} \right)=
N\sum\limits_{\lambda\in\mathcal{E}_{n}}\left|\left|c_{\lambda_{+}}\right|^{2}-\left|c_{\lambda}\right|^{2}\right|$, where $\lambda_{+}$ is the clockwise nearest neighbour of $\lambda$
on $\sqrt{n}\mathcal{S}^{1}$, see \eqref{eq:BasicVariablesV}.\\

$\widetilde{V}(\underline{v}) = \frac{[\underline{v}]_{\infty} \cdot V(\underline{v})}{A_{4}(\underline{v})}$, see \eqref{eq:alphaDef}.\\

$\Fc_{1}(n;T(n),\eta(n)) = \left\{f_{n}:\: \widetilde{V}(\underline{v})< \eta(n) \cdot \frac{T(n)}{\log T(n)} \right\}$, see \eqref{eq:F_1_def}.\\

$\Fc_{2}(n;T(n),\eta(n)) = \left\{f_{n}:\: [\underline{v}]_{\infty} < T(n)^{\eta(n)} \right\}$, see \eqref{eq:F_2_def}.\\

$\widehat{\lambda}=\lambda/\sqrt{n}$: the projection of $\lambda \in \mathcal{E}_n$ onto $\mathcal{S}^{d-1}.$\\

$\Delta\left(n\right)=\sup\limits_{0\le a\le b\le2\pi}\left|\frac{1}{N}\cdot \#\left\{ 1\le j\le N:\,\phi_{j}\in\left[a,b\right]\,\text{mod\,}2\pi\right\} -\frac{\left(b-a\right)}{2\pi}\right|$: the discrepancy of the angles $ \phi_j $ corresponding to the lattice points $ \Ec_{2;n} $, see \eqref{eq:Discrepancy_2d}.\\

Hypothesis $ \mathcal{D}(n,\epsilon) $ holds if $\Delta\left(n\right)\le \left(\log n\right)^{-\frac{1}{2}\log\frac{\pi}{2}+\epsilon}$, see \eqref{eq:D_n_epsilon}.\\

$\Delta_{3}\left(n\right)=\sup\limits_{\begin{subarray}{c}
	x\in\mathcal{S}^{2}\\
	0<r\le2
	\end{subarray}}\left|\frac{1}{N} \cdot \#\left\{ \lambda\in\mathcal{E}_{3;n}:\,\left|\widehat{\lambda}-x\right|\le r\right\} -\frac{r^{2}}{4}\right|$: the spherical cap discrepancy of the points $ \mathcal{E}_{3;n} $, see \eqref{eq:Discrepancy_3d}.\\

$\Sc_{n}(l) = \left\{ (\lambda^{1},\ldots,\lambda^{l})\in(\Ec_{n})^{l}:\: \sum\limits_{i=1}^{l}\lambda^{i}=0  \right\}$: the length-$ l $ spectral correlation set, see \eqref{eq:Sc correlations def}.\\

$\Dc_{n}(l) = \left\{ \pi(\lambda^{1},-\lambda^{1},\ldots,\lambda^{k},-\lambda^{k}): \lambda^{1},\ldots,\lambda^{k}\in (\Ec_{n})^{k},\,\pi\in S_{l}\right\}$: the diagonal correlations set, see \eqref{eq:Dc diag def}.\\

$\mathcal{A}_n (2k) = \left\{\left(\lambda_{1},\dots,\lambda_{2k}\right)\in\Dc_{n}(2k): \; \forall 1\le i\le k \; \lambda_{2i-1}\ne-\lambda_{2i} \right\}$: the set of ``admissible'' $ 2k $-tuples of lattice points, see \eqref{eq:admissible_tuples}.\\

$S\left(\lambda_{1},\dots,\lambda_{2k}\right)$: the structure set of an admissible $ 2k $-tuple $\left(\lambda_{1},\dots,\lambda_{2k}\right)$, see \eqref{eq:admissible_tuples}.\\

$J_{\alpha}\left(x\right)$: the Bessel function
of the first kind of order $\alpha$.\\

$g_{d}\left(x\right)=\frac{J_{d/2}\left(2 \pi x\right)}{(2 \pi x)^{d/2}}$:
the Fourier transform of the characteristic function of the unit
ball in $\mathbb{R}^{d}$, see \eqref{eq:g_d}.\\

$	h_{2}\left(x\right)=\frac{J_{1}\left(2 \pi x\right)^{2}}{(2\pi x)^{2}}$, see \eqref{eq:h_2}.\\

$h_{3}\left(x\right)=2\pi^{-1}(2 \pi x)^{-4}\left(\frac{\sin 2 \pi x}{2\pi x}-\cos 2 \pi x\right)^{2}$, see \eqref{eq:h_3}.\\

$F_{\lambda_{0}}\left(s\right)=\frac{1}{N}\cdot \#\left\{ \lambda\in\mathcal{E}_{2;n}:\,\left\Vert \widehat{\lambda}-\widehat{\lambda_{0}}\right\Vert \le s\right\}$, see \eqref{eq:F_Lambda_Def}.\\

$F\left(s\right)=F_{f_{n}}\left(s\right)=\sum\limits_{\begin{subarray}{c}
	\lambda,\lambda'\in\mathcal{E}_{2;n}\\
	0<\left\Vert \widehat{\lambda}-\widehat{\lambda'}\right\Vert \le s
	\end{subarray}}\left|c_{\lambda}\right|^{2}\left|c_{\lambda'}\right|^{2}$, see \eqref{eq:F_Function}.\\

$F_{3}\left(s\right)=\frac{1}{N^{2}} \cdot \#\left\{ \lambda\ne\lambda'\in\mathcal{E}_{3;n}:\left\Vert \widehat{\lambda}-\widehat{\lambda'}\right\Vert \le s\right\}$, see \eqref{eq:F_3}.\\

$\E_{x_{0},\rho}[X_{f_{n},r}] = \frac{1}{\vol(B_{x_0}(\rho))}\int\limits_{B_{x_{0}}(\rho)}X_{f_{n},r;x}dx$: the ``restricted'' expected value of  $X_{f_{n},r} $, see \eqref{eq:restricted_expectation}.\\

$\Vc_{x_{0},\rho}(X_{f_{n},r}) = \E_{x_{0},\rho}[(X_{f_{n},r}-\E_{x_{0},\rho}[X_{f_{n},r}])^{2}]$: the restricted variance of $X_{f_{n},r} $, see \eqref{eq:restricted_variance}.\\

$\mathcal{C}_{n}(l;K) =
\left\{(\lambda^{1},\ldots,\lambda^{l})\in\mathcal{E}_{n}^{l}:\: 0 < \left\| \sum\limits_{j=1}^{l}\lambda^{j} \right\| \le K  \right\}$:  the set of length-$l$ spectral quasi-correlations, see \eqref{eq:quasi_correlations}.\\

Hypothesis $\mathcal{A}(n;l,\delta)$ holds if
$\mathcal{C}_{n}(l;n^{1/2-\delta}) = \varnothing$, see \eqref{eq:sep_hypothesis}.

\section{Proof of Theorem \ref{thm:VarMainGeneralized}, part 1: asymptotics for the variance, $d=2$.}

\label{sec:Proof_Main_thm_part1}

\subsection{Expressing the variance}

We begin with some preliminary expressions for the variance. Note that if $x$
is drawn randomly, uniformly on $\mathbb{T}^{d}$, then
\begin{equation}
\mathbb{E}\left[X_{f_{n},r}\right]=\frac{\pi^{d/2}}{\Gamma\left(d/2+1\right)}r^{d},\label{eq:ExpectationEquality}
\end{equation}
and therefore in this case, we have
\begin{equation}
\label{eq:variance_integral_form}
\Vc(X_{f_{n},r}) = \int\limits_{\Tb^{d}}\left(\int_{B_{x}\left(r\right)}f_{n}\left(y\right)^{2}\,\text{d}y- \frac{\pi^{d/2}}{\Gamma\left(d/2+1\right)}r^{d} \right)^{2}dx.
\end{equation}
Let $J_{\alpha}\left(x\right)$ be the Bessel function
of the first kind of order $\alpha$. The following lemma, proved in section \ref{sec:AuxLemmasProof}, explicates the inner integral in \eqref{eq:variance_integral_form}:
\begin{lemma}
	\label{lem:InnerIntegral}We have
	\begin{align}
	& \int_{B_{x}\left(r\right)}f_{n}\left(y\right)^{2}\,\text{d}y-\frac{\pi^{d/2}}{\Gamma\left(d/2+1\right)}r^{d} =\left(2\pi\right)^{d/2}r^{d}\sum_{\begin{subarray}{c}
		\lambda,\lambda'\in\mathcal{E}_{n}\\
		\lambda\ne\lambda'
		\end{subarray}}c_{\lambda}\overline{c_{\lambda'}}e\left(\left\langle x,\lambda-\lambda'\right\rangle \right)g_{d}\left(r\left\Vert \lambda-\lambda'\right\Vert \right),\label{eq:IntegrandVar}
	\end{align}
	where
	\begin{equation}
	\label{eq:g_d}
	g_{d}\left(x\right):=\frac{J_{d/2}\left(2 \pi x\right)}{(2 \pi x)^{d/2}}
	\end{equation}
	is the Fourier transform of the characteristic function of the unit
	ball in $\mathbb{R}^{d}$.
\end{lemma}

The following formula for the variance follows from Lemma \ref{lem:InnerIntegral}, \eqref{eq:Zygmund 4-corr} and \eqref{eq:3d 4-corr}:

\begin{lemma}
	\label{lem:VarExpd2}~
	\begin{enumerate}
		\item (Granville-Wigman \cite[Lemma 2.1]{GranvilleWigman}) For $d=2$
		we have
		\begin{equation}
		\label{eq:VarFormula2d}
		\text{\ensuremath{\mathcal{V}}}\left(X_{f_{n},r}\right)=8\pi^{2}r^{4}\sum_{\begin{subarray}{c}
			\lambda,\lambda'\in\mathcal{E}_{n}\\
			\lambda\ne\lambda'
			\end{subarray}}\left|c_{\lambda}\right|^{2}\left|c_{\lambda'}\right|^{2}h_{2}\left(r\left\Vert \lambda-\lambda'\right\Vert \right)
		\end{equation}
		where
		\begin{equation}
		\label{eq:h_2}
		h_{2}\left(x\right):=\frac{J_{1}\left(2 \pi x\right)^{2}}{(2\pi x)^{2}}.
		\end{equation}
		\item For $d=3$ and for every $\epsilon>0$, we have
		\begin{align}
		\text{\ensuremath{\mathcal{V}}}\left(X_{f_{n},r}\right) & =16\pi^{3}r^{6}\sum_{\begin{subarray}{c}
			\lambda,\lambda'\in\mathcal{E}_{n}\\
			\lambda\ne\lambda'
			\end{subarray}}\left|c_{\lambda}\right|^{2}\left|c_{\lambda'}\right|^{2}h_{3}\left(r\left\Vert \lambda-\lambda'\right\Vert \right)\label{eq:VarFormula3d}\\
		& +O\left([\underline{v}]_{\infty}^2 r^{6}N^{-1/4+\epsilon}\right),\nonumber
		\end{align}
		where
		\begin{equation}
		\label{eq:h_3}
		h_{3}\left(x\right):=2\pi^{-1}(2 \pi x)^{-4}\left(\frac{\sin 2 \pi x}{2\pi x}-\cos 2 \pi x\right)^{2}.
		\end{equation}
	\end{enumerate}
\end{lemma}
Note that functions $g_{2}$ and $h_{2}$ satisfy the following properties:
\begin{lemma}[{\cite[(6.575.2), (8.440), (8.451.1), (8.472.2)]{GradshtaynRizhik}}]
	\label{lem:H2Formulas}We have
	\begin{enumerate}
		\item $\int_{0}^{\infty}h_{2}\left(s\right)\,\text{d}s=\frac{2}{3\pi^2}$.
		\item $g_{2}\left(s\right)\sim\frac{1}{2}\hspace{1em}\left(s\to0\right)$.
		\item $g_{2}\left(s\right)\ll s^{-3/2}\hspace{1em}\left(s\to\infty\right)$.
		\item $g_{2}'\left(s\right)=-\frac{J_{2}\left(2\pi s\right)}{s}\ll\left(1+s\right)^{-3/2}.$
	\end{enumerate}
\end{lemma}

\subsection{Proof of Theorem \ref{thm:VarMainGeneralized}, part 1:}

\begin{definition}
	For $\lambda\in\mathcal{E}_{n}$ let $\widehat{\lambda}=\lambda/\sqrt{n}$
	be the projection of $\lambda$ onto the unit circle $\mathcal{S}^{1}.$
	\begin{enumerate}
		\item For $ \lambda_0 \in \mathcal{E}_n $ and $0\le s\le2$, denote	\begin{equation}
		\label{eq:F_Lambda_Def}
		F_{\lambda_{0}}\left(s\right)=\frac{1}{N}\cdot \#\left\{ \lambda\in\mathcal{E}_{n}:\,\left\Vert \widehat{\lambda}-\widehat{\lambda_{0}}\right\Vert \le s\right\} .
		\end{equation}
	\item For $0\le s\le2$ denote \begin{equation}
	F\left(s\right)=F_{f_{n}}\left(s\right)=\sum_{\begin{subarray}{c}
		\lambda,\lambda'\in\mathcal{E}_{n}\\
		0<\left\Vert \widehat{\lambda}-\widehat{\lambda'}\right\Vert \le s
		\end{subarray}}\left|c_{\lambda}\right|^{2}\left|c_{\lambda'}\right|^{2}.\label{eq:F_Function}
	\end{equation}
\end{enumerate}
\end{definition}

	Recall that $\widetilde{V}(\underline{v})=\cos^{2}\theta \cdot [\underline{v}]_{\infty} V(\underline{v})$
by \eqref{eq:alphaDef} and \eqref{eq:A4<->theta}.

\begin{proposition}
	\label{prop:DiscreteProp}We have
	\begin{align*}
	F\left(s\right) & =\frac{s}{\pi\cos^{2}\theta}\left(1+O\left(s^{2}+s^{-1}\Delta\left(n\right)+\widetilde{V}(\underline{v}) s+\widetilde{V}(\underline{v}) s^{-1}\Delta\left(n\right)^{2}\right)\right).
	\end{align*}
\end{proposition}
We postpone the proof of Proposition \ref{prop:DiscreteProp} until section \ref{sec:ProofOfLemmaDiscrete} to present the proof of the first part of Theorem \ref{thm:VarMainGeneralized} (that yields the first part of Theorem \ref{thm:VarMain}):

\begin{proof}[Proof of Theorem \ref{thm:VarMainGeneralized}, part 1 assuming Proposition \ref{prop:DiscreteProp}]
	Assume that $ n\in S_2 $ satisfies the hypothesis $ \mathcal{D}(n,\epsilon/2) $. We may rewrite (\ref{eq:VarFormula2d}) as
	\begin{equation}
	\mathcal{V}\left(X_{f_{n},r}\right)=8\pi^{2}r^{4}\int_{0}^{2}h_{2}\left(Ts\right)\,\text{d}F\left(s\right).\label{eq:VarianceintegralForm}
	\end{equation}
	We apply integration by parts  to (\ref{eq:VarianceintegralForm})
	twice, in opposite directions: first, by integration by parts and
	Proposition \ref{prop:DiscreteProp}, we get
	\begin{align}
	8\pi^{2}r^{4}\int_{0}^{2}h_{2}\left(Ts\right)\,\text{d}F\left(s\right) & =8\pi^{2}r^{4}h_{2}\left(2T\right)F\left(2\right)-8\pi^{2}r^{4}\int_{0}^{2}F\left(s\right)\,\text{d}h_{2}\left(Ts\right)\label{eq:VarEqAfterIntByParts}\\
	& =8\pi^{2}r^{4}h_{2}\left(2T\right)F\left(2\right)-8\pi r^{4}\cos^{-2}\theta\int_{0}^{2}s\,\text{d}h_{2}\left(Ts\right)\nonumber \\
	& +Err\left(X_{f_n,r}\right)\nonumber
	\end{align}
	where
	\begin{align*}
	& Err\left(X_{f_n,r}\right)\ll r^{4}\cos^{-2}\theta\int_{0}^{2}\left(s^{3}+\Delta\left(n\right)+\widetilde{V}(\underline{v}) s^{2}+\widetilde{V}(\underline{v})\Delta\left(n\right)^{2}\right)T\left|h_{2}'\left(Ts\right)\right|\,\text{d}s.
	\end{align*}
	Integrating by parts again, the first two terms on the r.h.s of (\ref{eq:VarEqAfterIntByParts})
	satisfy
	\begin{align}
	& 8\pi^{2}r^{4}h_{2}\left(2T\right)F\left(2\right)-8\pi r^{4}\cos^{-2}\theta\int_{0}^{2}s\,\text{d}h_{2}\left(Ts\right)=8\pi^{2}r^{4}h_{2}\left(2T\right)F\left(2\right)\label{eq:MainTermsVar}\\
	& -16\pi r^{4}h_{2}\left(2T\right)\cos^{-2}\theta+8\pi r^{4}\cos^{-2}\theta\int_{0}^{2}h_{2}\left(Ts\right)\,\text{d}s.\nonumber
	\end{align}
	By the first and the third parts of Lemma \ref{lem:H2Formulas},
	\begin{equation}
	\int_{0}^{2}h_{2}\left(Ts\right)\,\text{d}s=\frac{1}{T}\int_{0}^{2T}h_{2}\left(s\right)\,\text{d}s=\frac{2}{3\pi^2}T^{-1}+O\left(T^{-3}\right),\label{eq:h2Integral}
	\end{equation}
	and therefore, substituting (\ref{eq:h2Integral}) into (\ref{eq:MainTermsVar}),
	we obtain
	\begin{align}
	\label{eq:VarianceMainTerms}
	8\pi^{2}r^{4}h_{2}\left(2T\right)F\left(2\right)-8\pi r^{4}\cos^{-2}\theta\int_{0}^{2}s\,\text{d}h_{2}\left(Ts\right) & =\frac{16}{3\pi}\cos^{-2}\theta r^{4}T^{-1} +O\left(\cos^{-2}\theta r^{4}T^{-3}\right).
	\end{align}
	By the fourth part of Lemma \ref{lem:H2Formulas},
	\begin{align*}
	\int_{0}^{2}T\left|h_{2}'\left(Ts\right)\right|\,\text{d}s & =\int_{0}^{2T}\left|h_{2}'\left(s\right)\right|\,\text{d}s\le\int_{0}^{\infty}\left|h_{2}'\left(s\right)\right|\,\text{d}s<\infty,
	\end{align*}
	\[
	\int_{0}^{2}s^{2}T\left|h_{2}'\left(Ts\right)\right|\,\text{d}s=T^{-2}\int_{0}^{2T}s^{2}\left|h_{2}'\left(s\right)\right|\,\text{d}s\ll T^{-2}\log T
	\]
	and
	\[
	\int_{0}^{2}s^{3}T\left|h_{2}'\left(Ts\right)\right|\,\text{d}s=T^{-3}\int_{0}^{2T}s^{3}\left|h_{2}'\left(s\right)\right|\,\text{d}s\ll T^{-2},
	\]
	and therefore for $ n $ satisfying $ \mathcal{D}(n,\epsilon/2), $
	\begin{align}
	Err\left(X_{f_n,r}\right) & \ll\cos^{-2}\theta r^{4}\left(T^{-2}+\Delta\left(n\right)+\widetilde{V}(\underline{v}) T^{-2}\log T+\widetilde{V}(\underline{v})\Delta\left(n\right)^{2}\right)\nonumber \\
	& \ll\cos^{-2}\theta r^{4}\left(T^{-2}+\left(\log n\right)^{-\frac{1}{2}\log\frac{\pi}{2}+\frac{\epsilon}{2}}+\widetilde{V}(\underline{v}) T^{-2}\log T+\widetilde{V}(\underline{v})\left(\log n\right)^{-\log\frac{\pi}{2}+\epsilon}\right),\label{eq:ErrorTerm}
	\end{align}
and \eqref{eq:var asympt d=2 precise} follows from \eqref{eq:VarEqAfterIntByParts}, \eqref{eq:VarianceMainTerms} and \eqref{eq:ErrorTerm}.
\end{proof}
Note that by (\ref{eq:ErrorTerm}), for Bourgain's eigenfunctions,
for almost all $n\in S_{2}$ we have
\begin{equation*}
\sup_{\begin{subarray}{c}
	r > r_{0} \\ f_n\in \Bc{n}
	\end{subarray}}\left|\frac{\text{\ensuremath{\mathcal{V}}}\left(X_{f_{n},r}\right)}{r^{4}}-\frac{16}{3\pi}T^{-1}\right|=O\left(T_{0}^{-2}+\left(\log n\right)^{-\frac{1}{2}\log\frac{\pi}{2}+\epsilon}\right)
\end{equation*}
for every $\epsilon>0$, and in particular
\begin{equation}
\label{eq:o_r4_2d}
\text{\ensuremath{\mathcal{V}}}\left(X_{f_{n},r}\right)=o\left(r^{4}\right)
\end{equation}
uniformly for $r > r_{0}$ for a density one sequence in $S_{2}$. Therefore, \eqref{eq:o_r4_2d} serves as a refinement of \cite[Corollary 1.10]{GranvilleWigman} for this specific case
(for a density one sequence in $S_{2}$), since \cite[Corollary 1.10]{GranvilleWigman} yields $\mathcal{V}\left(X_{f_n,r}\right)=o\left(r^{4}\right)$ under the additional assumption $T_{0}\gg n^{4\epsilon}$.

\subsection{Proof of Proposition \ref{prop:DiscreteProp}} \label{sec:ProofOfLemmaDiscrete}

In this section we prove Proposition \ref{prop:DiscreteProp}. First, we
define a binary relation on $\mathcal{E}_n$:
\begin{definition}
	for $\lambda\ne-\lambda'\in\mathcal{E}_{n}$, we say that $\lambda\prec\lambda'$
	if the arc on the circle $\sqrt{n}\mathcal{S}^{1}$ that connects
	$\lambda$ to $\lambda'$ counter-clockwise to $\lambda'$ is shorter
	than the arc that connects them clockwise to $\lambda'$. Recall that
	$\lambda_{+}$ is the clockwise nearest neighbour of $\lambda$ on
	$\sqrt{n}\mathcal{S}^{1}$.
	The proof of Proposition \ref{prop:DiscreteProp} employs the following
	auxiliary lemma to be proved at section \ref{sec:AuxLemmasProof}, establishing
	(\ref{eq:F_Function}) in the particular case $\left|c_{\lambda}\right|^{2}=1$
	for every $\lambda\in\mathcal{E}_{n}$:
\end{definition}
\begin{lemma}
	\label{lem:CosToDist}
	
	Fix $\lambda'\in\mathcal{E}_{n}.$ For $0\le s<2$, we have
	\begin{equation}
	\frac{1}{N}\cdot \#\left\{ \lambda\in\mathcal{E}_{n}:\,\lambda\succeq\lambda',\left\Vert \widehat{\lambda}-\widehat{\lambda'}\right\Vert \le s\right\} =\frac{s}{2\pi}+O\left(s^{3}+\Delta\left(n\right)\right)\label{eq:CosToDistEq}
	\end{equation}
	where the constant involved in the 'O'-notation in (\ref{eq:CosToDistEq})
	is absolute.
\end{lemma}
\begin{remark}
	The estimate (\ref{eq:CosToDistEq}) is also valid with either `$\succ$', `$\preceq$'
	or `$\prec$' in place of `$\succeq$'.
\end{remark}
We are now in a position to prove Proposition \ref{prop:DiscreteProp}:
\begin{proof}[Proof of Proposition \ref{prop:DiscreteProp} assuming Lemma \ref{lem:CosToDist}]
	First, we write
	\begin{align}
	F\left(s\right) & =\sum_{\lambda'\in\mathcal{E}_{n}}\left|c_{\lambda'}\right|^{2}\sum_{\begin{subarray}{c}
		\lambda\in\mathcal{E}_{n}\\
		\left\Vert \widehat{\lambda}-\widehat{\lambda'}\right\Vert \le s\\
		\lambda\preceq\lambda'
		\end{subarray}}\left|c_{\lambda}\right|^{2}+\sum_{\lambda'\in\mathcal{E}_{n}}\left|c_{\lambda'}\right|^{2}\sum_{\begin{subarray}{c}
		\lambda\in\mathcal{E}_{n}\\
		\left\Vert \widehat{\lambda}-\widehat{\lambda'}\right\Vert \le s\\
		\lambda\succeq\lambda'
		\end{subarray}}\left|c_{\lambda}\right|^{2}\label{eq:PartialSummation} +O\left(\frac{A_{4}}{N}\right).
	\end{align}
	Using summation by parts, we get that for every $\lambda'\in\mathcal{E}_{n}$
	\begin{align}
	\sum_{\begin{subarray}{c}
		\lambda\in\mathcal{E}_{n}\\
		\left\Vert \widehat{\lambda}-\widehat{\lambda'}\right\Vert \le s\\
		\lambda\preceq\lambda'
		\end{subarray}}\left|c_{\lambda}\right|^{2} & =\left|c_{\lambda'}\right|^{2} \cdot \#\left\{ \lambda\in\mathcal{E}_{n}:\,\lambda\preceq\lambda',\left\Vert \widehat{\lambda}-\widehat{\lambda'}\right\Vert \le s\right\} \label{eq:PartialSumEq}\\
	& -\sum_{\begin{subarray}{c}
		\lambda\in\mathcal{E}_{n}\\
		\left\Vert \widehat{\lambda}-\widehat{\lambda'}\right\Vert \le s\\
		\lambda\prec\lambda'
		\end{subarray}}\left(\left|c_{\lambda_{+}}\right|^{2}-\left|c_{\lambda}\right|^{2}\right) \cdot \#\left\{ \mu\in\mathcal{E}_{n}:\,\mu\preceq\lambda,\,\left\Vert \widehat{\mu}-\widehat{\lambda'}\right\Vert \le s\right\} .\nonumber
	\end{align}
	By Lemma \ref{lem:CosToDist}, the contribution of the first term
	on the r.h.s of (\ref{eq:PartialSumEq}) to $F\left(s\right)$ is
	\begin{align}
	\label{eq:first_term_discrete}
	 \sum_{\lambda'\in\mathcal{E}_{n}}\left|c_{\lambda'}\right|^{4} \cdot \#\left\{ \lambda\in\mathcal{E}_{n}:\,\lambda\preceq\lambda',\left\Vert \widehat{\lambda}-\widehat{\lambda'}\right\Vert \le s\right\}
	 =A_{4} \cdot
	\left(s/2\pi+O\left(s^{3}+\Delta\left(n\right)\right)\right).
	\end{align}
	The contribution of the sum on the r.h.s of (\ref{eq:PartialSumEq})
	to $F\left(s\right)$ is
	\begin{align}
	\label{eq:second_term_discrete}
	& \sum_{\lambda'\in\mathcal{E}_{n}}\left|c_{\lambda'}\right|^{2}\sum_{\begin{subarray}{c}
		\lambda\in\mathcal{E}_{n}\\
		\left\Vert \widehat{\lambda}-\widehat{\lambda'}\right\Vert \le s\\
		\lambda\prec\lambda'
		\end{subarray}}\left(\left|c_{\lambda_{+}}\right|^{2}-\left|c_{\lambda}\right|^{2}\right) \cdot \#\left\{ \mu\in\mathcal{E}_{n}:\,\mu\preceq\lambda,\,\left\Vert \widehat{\mu}-\widehat{\lambda'}\right\Vert \le s\right\} \\
	& \ll N\left(s+\Delta\left(n\right)\right)\sum_{\begin{subarray}{c}
		\lambda\in\mathcal{E}_{n}\end{subarray}}\left|\left|c_{\lambda_{+}}\right|^{2}-\left|c_{\lambda}\right|^{2}\right|\sum_{\begin{subarray}{c}
		\lambda'\in\mathcal{E}_{n} \nonumber \\
		\left\Vert \widehat{\lambda}-\widehat{\lambda'}\right\Vert \le s\\
		\lambda\prec\lambda'
		\end{subarray}}\left|c_{\lambda'}\right|^{2}\ll\left(s+\Delta\left(n\right)\right)^{2}[\underline{v}]_{\infty} V(\underline{v}) .
	\end{align}
	By \eqref{eq:first_term_discrete} and \eqref{eq:second_term_discrete}, we have
	\begin{equation}
	\label{eq:first_summation_discrete}
	\sum_{\begin{subarray}{c}
		\lambda\in\mathcal{E}_{n}\\
		\left\Vert \widehat{\lambda}-\widehat{\lambda'}\right\Vert \le s\\
		\lambda\preceq\lambda'
		\end{subarray}}\left|c_{\lambda}\right|^{2} = A_{4} \cdot
	\left(s/2\pi+O\left(s^{3}+\Delta\left(n\right)\right)\right) + O\left(\left(s+\Delta\left(n\right)\right)^{2}[\underline{v}]_{\infty}V(\underline{v}) \right).
	\end{equation}
	By symmetry, the second summation in (\ref{eq:PartialSummation})
	obeys \eqref{eq:PartialSumEq} with `$ \succ $', `$ \succeq $' and $ |c_{\lambda_-}|^2 $ in place of `$ \prec $', `$ \preceq $'  and $ |c_{\lambda_+}|^2 $, where  $ \lambda_- $ is the counter-clockwise nearest neighbour to $ \lambda $. The statement of Proposition \ref{prop:DiscreteProp}
	follows from the analogues of the estimates \eqref{eq:first_term_discrete}, \eqref{eq:second_term_discrete} and \eqref{eq:first_summation_discrete}.
\end{proof}

\section{Proof of Theorem \ref{thm:VarMainGeneralized}, part 2: Gaussian moments, $d=2$.}

\label{sec:Proof_Main_thm_part2}

In this section we study the higher moments of $\hat{X}_{f_{n},r}$ defined in \eqref{eq:standardizedX}, and prove the second part of
Theorem \ref{thm:VarMainGeneralized}, also implying the second part of Theorem \ref{thm:VarMain}.

The proof of the following lower bound for $ \Vc\left(X_{f_{n},r}\right) $ with  $f_{n}\in \Fc_{2}(n;T(n),\eta(n))$ goes along the same lines as the proof of the lower bound in Theorem \ref{thm:UpperBound2d} below:

\begin{lemma}
	\label{lem:Lower_Bound_F2}
	In the setting of Theorem \ref{thm:VarMainGeneralized} part (2), we have
	 \begin{equation*}
     \frac{\Vc(X_{f_n,r})}{r^{4}} \gg T(n)^{-1-2\eta(n)}
	 \end{equation*}
	 uniformly for $r_{0} < r < r_1$ and
	 $f_{n}\in \Fc_{2}(n;T(n),\eta(n))$.
	
\end{lemma}

Before proceeding to the proof of Theorem \ref{thm:VarMainGeneralized},
we introduce some notation:

\begin{definition} \hspace{1em}
	\begin{enumerate}
		\item 	Define the set of ``admissible'' $2k$-tuples of lattice points by
		\begin{equation}
		\label{eq:admissible_tuples}
		\mathcal{A}_n (2k) = \left\{\left(\lambda_{1},\dots,\lambda_{2k}\right)\in\Dc_{n}(2k): \; \forall 1\le i\le k \; \lambda_{2i-1}\ne-\lambda_{2i} \right\}.
		\end{equation}		
		\item 	Given an admissible $2k$-tuple of lattice points $\left(\lambda_{1},\dots,\lambda_{2k}\right)\in \mathcal{A}_n (2k)$,
		let $\sim$ be the equivalence relation on the set $\left\{ 1,\dots,2k\right\} $,
		generated by:
		\begin{enumerate}
			\item $2i-1\sim2i$ for every $1\le i\le k$.
			\item $j\sim j'$ if $\lambda_{j}+\lambda_{j}'=0.$
		\end{enumerate}
		Let $\left\{ \Lambda_{1},\dots,\Lambda_{m}\right\} $ be the partition
		of $\left\{ 1,\dots,2k\right\} $ into equivalence classes of $\sim$,
		and denote $l_{j}=\#\Lambda_{m}/2$ for  $1\le j\le m$, so that
		$\sum_{j=1}^{m}l_{j}=k$; clearly, $2\le l_{j}\in\mathbb{Z}$ for
		every $1\le j\le m$. We call the multiset
		\begin{equation}
		\label{eq:structure_set}
		S\left(\lambda_{1},\dots,\lambda_{2k}\right):=\left\{ l_{1},\dots,l_{m}\right\}
		\end{equation}
		the structure set of the $2k$-tuple $\left(\lambda_{1},\dots,\lambda_{2k}\right).$
	\end{enumerate}

\end{definition}
Recall that the moments of a standard Gaussian random variable  $Z\sim N(0,1)$ are
\[
\E[Z^{k}]=\begin{cases}
\left(k-1\right)!! & k\,\text{even}\\
0 & k\,\text{odd}.
\end{cases}
\] We are now in a position to prove the second part of Theorem \ref{thm:VarMainGeneralized}.
\begin{proof}[Proof of Theorem \ref{thm:VarMainGeneralized}, part 2]
	By the length-$2k$ diagonal domination
	assumption, we have
	\begin{align}
	\E[\hat{X}_{f_{n},r}^{k}] & =(2\pi)^k r^{2k}\Vc\left(X_{f_{n},r}\right)^{-k/2}\sum_{\begin{subarray}{c}
		\left(\lambda_{1},\dots,\lambda_{2k}\right)\in \mathcal{A}_n (2k)\end{subarray}}\label{eq:kthMoment}\prod_{j=1}^{k}c_{\lambda_{2j-1}}c_{\lambda_{2j}}g_{2}\left(r\left\Vert \lambda_{2j-1}+\lambda_{2j}\right\Vert \right) \\ &+O\left(\Vc\left(X_{f_{n},r}\right)^{-k/2} [\underline{v}]_{\infty}^k r^{2k}N^{-\gamma}\right)\nonumber
	\end{align} for some $ \gamma>0. $
	We can rearrange the summation in (\ref{eq:kthMoment}), first summing over
	all possible structure sets $\mathcal{L}=\left\{ l_{1},\dots,l_{m}\right\} $ and then
	summing over the admissible $2k$-tuples $\left(\lambda_{1},\dots,\lambda_{2k}\right)\in\mathcal{E}_{n}^{2k}$
	with the given structure set $S\left(\lambda_{1},\dots,\lambda_{2k}\right)=\mathcal{L} $: let
	\begin{align*}
	S_{\mathcal{L}} & :=\sum_{\begin{subarray}{c}
		\begin{subarray}{c}
		\left(\lambda_{1},\dots,\lambda_{2k}\right)\in \mathcal{A}_n (2k)\\
		S\left(\lambda_{1},\dots,\lambda_{2k}\right)=\mathcal{L}
		\end{subarray}\end{subarray}}\prod_{j=1}^{k}c_{\lambda_{2j-1}}c_{\lambda_{2j}}g_{2}\left(r\left\Vert \lambda_{2j-1}+\lambda_{2j}\right\Vert \right),
	\end{align*}
	so that we may rewrite the summation on the r.h.s. of \eqref{eq:kthMoment} as
	\begin{align}
	\label{eq:Inner_Sum_with_SL}
	& \sum_{\begin{subarray}{c}
		\left(\lambda_{1},\dots,\lambda_{2k}\right)\in \mathcal{A}_n (2k)\end{subarray}}\prod_{j=1}^{k}c_{\lambda_{2j-1}}c_{\lambda_{2j}}g_{2}\left(r\left\Vert \lambda_{2j-1}+\lambda_{2j}\right\Vert \right)=\sum_{\begin{subarray}{c}
		l_{1}+\dots+l_{m}=k\\
		l_{1},\dots,l_{m}\ge2
		\end{subarray}}\sum_{\begin{subarray}{c}
		\begin{subarray}{c}
		\left(\lambda_{1},\dots,\lambda_{2k}\right)\in \mathcal{A}_n (2k)\\
		S\left(\lambda_{1},\dots,\lambda_{2k}\right)=\mathcal{L}
		\end{subarray}\end{subarray}}S_{\mathcal{L}}.
	\end{align}
	For a fixed structure set $\mathcal{L}=\left\{ l_{1},\dots,l_{m}\right\} $, we have
	\begin{align}
	S_{\mathcal{L}} =a\left(\mathcal{L}\right)\prod_{j=1}^{m}\sum_{\lambda_{1},\dots,\lambda_{l_{j}}\in\mathcal{E}_{n}}\left|c_{\lambda_{1}}\right|^{2}g_{2}\left(r\left\Vert \lambda_{l_{j}}-\lambda_{1}\right\Vert \right)\label{eq:RearrangeInPartition}\prod_{i=1}^{l_{j}-1}\left|c_{\lambda_{i+1}}\right|^{2}g_{2}\left(r\left\Vert \lambda_{i}-\lambda_{i+1}\right\Vert \right)+O\left([\underline{v}]_{\infty}^k N^{-1}\right)
	\end{align}
	where $a\left(\mathcal{L}\right)$ is a constant depending on $\mathcal{L}$; omitting the condition that the lattice points are distinct on the
	r.h.s of (\ref{eq:RearrangeInPartition}) is absorbed within the error term in (\ref{eq:RearrangeInPartition}). Thus,
	\begin{align}
	\label{eq:SL_upper_bound}
	S_{\mathcal{L}} & \ll [\underline{v}]_{\infty}^k N^{-k}\prod_{j=1}^{m}\sum_{\lambda_{1}\in\mathcal{E}_{n}}\sum_{\lambda_{2}\in\mathcal{E}_{n}}\left|g_{2}\left(r\left\Vert \lambda_{2}-\lambda_{1}\right\Vert \right)\right|\cdots\sum_{\lambda_{l_{j}}\in\mathcal{E}_{n}}\left|g_{2}\left(r\left\Vert \lambda_{l_{j}-1}-\lambda_{l_{j}}\right\Vert \right)\right|+[\underline{v}]_{\infty}^k N^{-1}.
	\end{align}
	Recall the definition of $ F_{\lambda_0} $ in \eqref{eq:F_Lambda_Def}.
	By Lemma \ref{lem:CosToDist}, we have
	\begin{equation}
	\label{eq:f_lambda_zero}
	F_{\lambda_{0}}\left(s\right)=\frac{s}{\pi}+O\left(s^{3}+\Delta\left(n\right)\right)=O\left(s+\Delta\left(n\right)\right).
	\end{equation}
	Thus, by Lemma \ref{lem:H2Formulas} and  \eqref{eq:f_lambda_zero}, we have that
	\begin{align}
	\label{eq:abs_g2_bound}
	\frac{1}{N}\sum_{\lambda\in\mathcal{E}_{n}}\left|g_{2}\left(r\left\Vert \lambda-\lambda_{0}\right\Vert \right)\right| & =\int_{0}^{2}\left|g_{2}\left(Ts\right)\right|\,\text{d}F_{\lambda_{0}}\left(s\right)\\
	& =\left|g_{2}\left(2T\right)\right|-\frac{1}{2N}+O\left(\int_{0}^{2}\left(s+\Delta\left(n\right)\right)T\left|g_{2}'\left(Ts\right)\right|\,\text{d}s\right) \nonumber \\
	& =O\left(T^{-3/2}+\left(\Delta\left(n\right)+T^{-1}\right)\int_{0}^{2T}\left|g_{2}'\left(s\right)\right|\,\text{d}s\right) \nonumber \\
	& =O\left(T^{-1}\right) \nonumber
	\end{align}
	for $ n $ satisfying the hypothesis $ \mathcal{D}(n,\epsilon) $. Applying \eqref{eq:abs_g2_bound} to each of the $l_{j}-1$ inner summations in \eqref{eq:SL_upper_bound}, we obtain
	
	\begin{align*}
	S_{\mathcal{L}} & \ll [\underline{v}]_{\infty}^k N^{-k+m}\prod_{j=1}^{m}\left(NT^{-1}\right)^{l_{j}-1} + [\underline{v}]_{\infty}^k N^{-1} \ll [\underline{v}]_{\infty}^k T^{-k+m}.
	\end{align*}
	Let $ \mathcal{L}_0 = \left\{ 2,2,\dots2,\right\} $. Note that if $\mathcal{L}\ne \mathcal{L}_0 $ then $m\le\frac{k-1}{2}$
	and therefore
	\begin{equation}
	\label{eq:SL_Estimate}
		S_{\mathcal{L}}=O\left([\underline{v}]_{\infty}^k T^{-\frac{k+1}{2}}\right).
	\end{equation}
	If $\mathcal{\mathcal{L}}=\mathcal{L}_0 $ (this is a viable option for $k$ even),
	then
	\begin{equation}
	\label{eq:SL_formula}
	S_{\mathcal{L}_0}=2^{k/2}\left(k-1\right)!!\left[\sum_{\lambda_{1}\ne\lambda_{2}\in\mathcal{E}_{n}}\left|c_{\lambda_{1}}\right|^{2}\left|c_{\lambda_{2}}\right|^{2}h_{2}\left(r\left\Vert \lambda_{1}-\lambda_{2}\right\Vert \right)\right]^{k/2}+O\left([\underline{v}]_{\infty}^k N^{-1}\right).
	\end{equation}
	By (\ref{eq:VarFormula2d}),
	\begin{equation}
		\label{eq:SL_Variance_asympt}
	\sum_{\lambda_{1}\ne\lambda_{2}\in\mathcal{E}_{n}}\left|c_{\lambda_{1}}\right|^{2}\left|c_{\lambda_{2}}\right|^{2}h_{2}\left(r\left\Vert \lambda_{1}-\lambda_{2}\right\Vert \right)=\frac{\Vc\left(X_{f_{n},r}\right)}{8\pi^2 r^4}.
	\end{equation}
Hence, \eqref{eq:SL_formula} and \eqref{eq:SL_Variance_asympt} yield
	\begin{equation}
	\label{eq:SL_final_form}
	S_{\mathcal{L}_0}=\left(k-1\right)!!\left(\frac{\Vc\left(X_{f_{n},r}\right)}{4\pi^2r^4}\right)^{k/2}+O\left([\underline{v}]_{\infty}^k N^{-1}\right).
	\end{equation}
	Substituting  \eqref{eq:SL_Estimate} and \eqref{eq:SL_final_form} into \eqref{eq:Inner_Sum_with_SL} and applying Lemma \ref{lem:Lower_Bound_F2}, we finally obtain that for $k$ even
	\begin{align*}
	\left|\E[\hat{X}_{f_{n},r}^{k}]-\left(k-1\right)!!\right| & \ll T^{k\eta(n)} [\underline{v}]_{\infty}^k\left(T^{-1/2}+T^{k/2}N^{-\min\{1,\gamma \}}\right) \ll T^{-1/2+2k\eta(n)}
	\end{align*}
	and since for $k$ odd $ \mathcal{L}=\mathcal{L}_0 $ is not a viable option, we obtain
	\begin{align*}
\E[\hat{X}_{f_{n},r}^{k}] & \ll T^{k\eta(n)} [\underline{v}]_{\infty}^k\left(T^{-1/2}+T^{k/2}N^{-\min\{1,\gamma \}}\right) \ll  T^{-1/2+2k\eta(n) },
	\end{align*}
	and the second part of Theorem \ref{thm:VarMainGeneralized} follows.
\end{proof}

\section{Proof of Theorem \ref{thm:Var3D}: asymptotics for the variance,
	$d=3$}

\label{sec:Proof_3d_theorem}
\subsection{Proof of Theorem \ref{thm:Var3D}}

Denote
\begin{equation}
\label{eq:F_3}
F_{3}\left(s\right)=\frac{1}{N^{2}} \cdot \#\left\{ \lambda\ne\lambda'\in\mathcal{E}_{n}:\left\Vert \widehat{\lambda}-\widehat{\lambda'}\right\Vert \le s\right\}
\end{equation}(cf. \eqref{eq:F_Function}), and
recall that the spherical cap discrepancy for the points in $\mathcal{E}_{n}$
is defined by
\begin{equation}
\label{eq:Discrepancy_3d}
\Delta_{3}\left(n\right)=\sup_{\begin{subarray}{c}
	x\in\mathcal{S}^{2}\\
	0<r\le2
	\end{subarray}}\left|\frac{1}{N} \cdot \#\left\{ \lambda\in\mathcal{E}_{n}:\,\left|\widehat{\lambda}-x\right|\le r\right\} -\frac{r^{2}}{4}\right|.
\end{equation}

\begin{lemma}
	\label{cor:DistMeasure3d}We have
	\begin{equation}
	\label{eq:3d_Discrepancy}
	F_{3}\left(s\right)=\frac{s^{2}}{4}+O\left(\Delta_{3}\left(n\right)\right).
	\end{equation}
\end{lemma}
\begin{proof}
	The estimate \eqref{eq:3d_Discrepancy} follows immediately from the definition of spherical cap discrepancy,
	since
	\begin{align*}
	F_{3}\left(s\right) & =\frac{1}{N}\sum_{\lambda'\in\mathcal{E}_{n}}\#\left\{ \lambda\in\mathcal{E}_{n}:\,0<\left\Vert \widehat{\lambda}-\widehat{\lambda'}\right\Vert \le s\right\}  =\frac{s^{2}}{4}+O\left(\Delta_{3}\left(n\right)\right).
	\end{align*}
\end{proof}
The discrepancy $\Delta_{3}\left(n\right)$ satisfies $ \Delta_3(n) \le n^{-\eta} $ for some small $\eta>0$, see \cite{BourgainRudnickSarnak}. We are now in a position to prove
Theorem \ref{thm:Var3D}:
\begin{proof}[Proof of Theorem \ref{thm:Var3D}]
	By \eqref{eq:VarFormula3d} we have
	\begin{align}
	\label{eq:var_3d_asymp_formula}
	\text{\ensuremath{\mathcal{V}}}\left(X_{f_n,r}\right) & =16\pi^{3}r^{6}\frac{1}{N^{2}}\sum_{\begin{subarray}{c}
		\lambda,\lambda'\in\mathcal{E}_{n}\\
		\lambda\ne\lambda'
		\end{subarray}}h_{3}\left(T\left\Vert \widehat{\lambda}-\widehat{\lambda'}\right\Vert \right)+O\left(r^{6}N^{-1/4+\epsilon}\right).
	\end{align}
	For the summation in \eqref{eq:var_3d_asymp_formula} we have,
	\[
	\frac{1}{N^{2}}\sum_{\begin{subarray}{c}
		\lambda,\lambda'\in\mathcal{E}_{n}\\
		\lambda\ne\lambda'
		\end{subarray}}h_{3}\left(T\left\Vert \widehat{\lambda}-\widehat{\lambda'}\right\Vert \right)=\int_{0}^{2}h_{3}\left(Ts\right)\,\text{d}F_{3}\left(s\right).
	\]
	Thus, integrating by parts and using Lemma \ref{cor:DistMeasure3d},
	\begin{align}
	\frac{1}{N^{2}}\sum_{\begin{subarray}{c}
		\lambda,\lambda'\in\mathcal{E}_{n}\\
		\lambda\ne\lambda'
		\end{subarray}}h_{3}\left(T\left\Vert \widehat{\lambda}-\widehat{\lambda'}\right\Vert \right) & =h_{3}\left(2T\right)F_{3}\left(2\right)-\int_{0}^{2}F_{3}\left(s\right)\,\text{d}h_{3}\left(Ts\right)\label{eq:VarAftIntParts3D}\\
	& =h_{3}\left(2T\right)F_{3}\left(2\right)-\frac{1}{4}\int_{0}^{2}s^{2}\,\text{d}h_{3}\left(Ts\right)+Err\left(X_{f_n,r}\right)\nonumber
	\end{align}
	where
	\begin{equation}
	\label{eq:error_term_3d}
	Err\left(X_{f_n,r}\right)\ll\Delta_{3}\left(n\right)\int_{0}^{2}T\left|h_{3}'\left(Ts\right)\right|\,\text{d}s.
	\end{equation}
	Note that $h_{3}\left(s\right)\ll s^{-4}$ as $s\to\infty.$ Thus,
	integrating by parts, the main term on the r.h.s of (\ref{eq:VarAftIntParts3D})
	satisfies
	\begin{equation}
	\label{eq:main_terms_after_intbyparts_3d}
	h_{3}\left(2T\right)F_{3}\left(2\right)-\frac{1}{4}\int_{0}^{2}s^{2}\,\text{d}h_{3}\left(Ts\right)=\frac{1}{2}\int_{0}^{2}s \cdot h_{3}\left(Ts\right)\,\text{d}s+O\left(T^{-4}\right),
	\end{equation}
so that
	\begin{equation}
	\label{eq:main_int_after_intbyparts}
	\int_{0}^{2}s \cdot h_{3}\left(Ts\right)\,\text{d}s=\frac{1}{T^{2}}\int_{0}^{2T}s \cdot h_{3}\left(s\right)\,\text{d}s=\frac{1}{T^{2}}\int_{0}^{\infty}s \cdot h_{3}\left(s\right)\,\text{d}s+O\left(T^{-4}\right).
	\end{equation}
	A direct computation shows that
	\begin{equation}
	\label{eq:final_calc_main_term_3d}
	\int_{0}^{\infty}s \cdot h_{3}\left(s\right)\,\text{d}s=\frac{1}{2\pi^3}\int_{0}^{\infty}\frac{1}{s^{3}}\left(\frac{\sin s}{s}-\cos s\right)^{2}\,\text{d}s=\left(2\pi\right)^{-3},
	\end{equation}
	and therefore, substituting \eqref{eq:final_calc_main_term_3d} into \eqref{eq:main_int_after_intbyparts} and then into \eqref{eq:main_terms_after_intbyparts_3d} we get
	\begin{equation}
	\label{eq:main_term_final_form_3d}
	h_{3}\left(2T\right)F_{3}\left(2\right)-\frac{1}{4}\int_{0}^{2}s^{2}\,\text{d}h_{3}\left(Ts\right)=\frac{1}{16\pi^3}T^{-2}+O\left(T^{-4}\right).
	\end{equation}
	Note that $h_{3}'\left(s\right)\ll\left(1+s^{4}\right)^{-1}$. Thus,
	\begin{equation}
	\label{eq:err_upper_bound_3d}
	\int_{0}^{2}T\left|h_{3}'\left(Ts\right)\right|\,\text{d}s=\int_{0}^{2T}\left|h_{3}'\left(s\right)\right|\,\text{d}s\le\int_{0}^{\infty}\left|h_{3}'\left(s\right)\right|\,\text{d}s<\infty
	\end{equation}
	and therefore, substituting \eqref{eq:err_upper_bound_3d} into \eqref{eq:error_term_3d} we obtain
	\begin{equation}
	Err\left(X_{f_n,r}\right)=O\left(\Delta_{3}\left(n\right)\right).\label{eq:3DErrorVar}
	\end{equation}
	Substituting \eqref{eq:3DErrorVar} into \eqref{eq:main_term_final_form_3d}  and finally into \eqref{eq:VarAftIntParts3D} we obtain (\ref{eq:AympVar3D}).
\end{proof}
Note that by (\ref{eq:3DErrorVar}),
\begin{equation*}
\sup_{\begin{subarray}{c}
	r > r_{0} \\ f_n\in \Bc{n}
	\end{subarray}}\left|\frac{\text{\ensuremath{\mathcal{V}}}\left(X_{f_{n},r}\right)}{r^{6}}-T^{-2}\right|=O\left(T_{0}^{-4}+n^{-\eta}\right)
\end{equation*}
for every $n\not\equiv0,4,7\,\left(8\right)$, and in particular
\[
\text{\ensuremath{\mathcal{V}}}\left(X_{f_{n},r}\right)=o\left(r^{6}\right)
\]
uniformly for $r > r_{0}$ for every $n\not\equiv0,4,7\,\left(8\right)$.

\section{Proofs of Theorem \ref{thm:UpperBound2d} and Theorem \ref{thm:UpperBound3d}}
\label{sec:ProofOfBoundsThm}
\begin{proof}[Proof of theorems \ref{thm:UpperBound2d} and \ref{thm:UpperBound3d}, upper bounds]
	
By substituting
the bound $\left|c_{\lambda}\right|^{2}\le N^{-1+\epsilon}$ into (\ref{eq:VarFormula2d}), we have

\begin{align}
\label{eq:var_ultraflat_bound}
\text{\ensuremath{\mathcal{V}}}\left(X_{f_{n},r}\right) & \ll r^{4} N^{-1+\epsilon}\sum_{\lambda_0 \in \mathcal{E}_n}\left|c_{\lambda_0}\right|^{2}\sum_{\lambda \in \mathcal{E}_n} h_{2}\left(r\left\Vert \lambda-\lambda_0\right\Vert \right).
\end{align}
By Lemma \ref{lem:H2Formulas} and by \eqref{eq:f_lambda_zero}, we have
\begin{align}
\label{eq:H2_sum}
\frac{1}{N}\sum_{\lambda\in\mathcal{E}_{n}}h_{2}\left(r\left\Vert \lambda-\lambda_{0}\right\Vert \right) & =\int_{0}^{2}h_{2}\left(Ts\right)\,\text{d}F_{\lambda_{0}}\left(s\right)\\
& =h_{2}\left(2T\right)-\frac{1}{4N}+O\left(\int_{0}^{2}\left(s+\Delta\left(n\right)\right)T\left|h_{2}'\left(Ts\right)\right|\,\text{d}s\right) \nonumber \\
& =O\left(T^{-1}+\left(\log n\right)^{-\frac{1}{2}\log\frac{\pi}{2}+\epsilon}\right) \nonumber
\end{align}
for $ n $ satisfying the hypothesis $ \mathcal{D}(n,\epsilon) $. Substituting \eqref{eq:H2_sum} in \eqref{eq:var_ultraflat_bound}, we get the upper bound in Theorem  \ref{thm:UpperBound2d}. The upper bound \eqref{eq:bounds var ultraflat d=3} in Theorem \ref{thm:UpperBound3d} follows along similar lines.
\end{proof}

We now turn to proving the claimed lower bounds for the variance of $X_{f_n,r}$. First, we need
the following lemma, proved at the end of section \ref{sec:ProofOfBoundsThm}:
\begin{lemma}
	\label{lem:LowerBoundPairs}
	\begin{enumerate}
		\item Let $\left\{ x_{m}\right\} _{m=1}^{M}$ be $M$ points on the unit
		circle $\mathcal{S}^{1}.$ For every $1<T<M/2$ we have
		\[
		\#\left\{ x_{i}\ne x_{j}:\,\left|x_{i}-x_{j}\right|\le1/T\right\} \gg M^{2}/T.
		\]
		\item Let $\left\{ x_{m}\right\} _{m=1}^{M}$ be $M$ points on the unit
		sphere $S^{2}.$ For every $1<T<\sqrt{M}/2$ we have
		\[
		\#\left\{ x_{i}\ne x_{j}:\,\left|x_{i}-x_{j}\right|\le1/T\right\} \gg M^{2}/T^{2}.
		\]
	\end{enumerate}
\end{lemma}
We are now in a position to prove the lower bounds \eqref{eq:bounds var ultraflat d=2}, \eqref{eq:bounds var ultraflat d=3} of Theorem \ref{thm:UpperBound2d} and Theorem \ref{thm:UpperBound3d}:
\begin{proof}[Proof of Theorem \ref{thm:UpperBound2d} and Theorem \ref{thm:UpperBound3d}, lower bounds assuming Lemma \ref{lem:LowerBoundPairs}]
	For $d=2$, we let
	\[
	R=\#\left\{ \lambda\in\mathcal{E}_{n}:\,\left|c_{\lambda}\right|^{2}\ge \frac{1}{2N}\right\},
	\]
	so that
	\[
	1=\sum_{\lambda\in\mathcal{E}_{n}}\left|c_{\lambda}\right|^{2}=\sum_{\lambda\in R}\left|c_{\lambda}\right|^{2}+\sum_{\lambda\notin R}\left|c_{\lambda}\right|^{2}\le N^{-1+\epsilon} \cdot \#R+1/2,
	\]
	and hence $\#R\ge 2N^{1-\epsilon}.$ By the second part of Lemma \ref{lem:H2Formulas},
	for $c>0$ sufficiently small we have
	\begin{align*}
	\Vc(X_{f_n,r}) & =8\pi^{2}r^{4}\sum_{\begin{subarray}{c}
		\lambda,\lambda'\in\mathcal{E}_{n}\\
		\lambda\ne\lambda'
		\end{subarray}}\left|c_{\lambda}\right|^{2}\left|c_{\lambda'}\right|^{2}h_{2}\left(T\left\Vert \widehat{\lambda}-\widehat{\lambda'}\right\Vert \right) \gg r^{4}N^{-2}\sum_{\lambda\ne\lambda'\in R}h_{2}\left(T\left\Vert \widehat{\lambda}-\widehat{\lambda'}\right\Vert \right)\\
	& \gg r^{4}N^{-2}\cdot \#\left\{ \lambda\ne\lambda'\in R:\,\left\Vert \widehat{\lambda}-\widehat{\lambda'}\right\Vert \le c/T\right\} .
	\end{align*}
	By the first part of Lemma \ref{lem:LowerBoundPairs},
	\[
	\Vc(X_{f_n,r})\gg r^{4}N^{-2}\left(\#R\right)^{2}T^{-1}\gg r^{4}N^{-2\epsilon}T^{-1}.
	\]
	
	The lower bound \eqref{eq:bounds var ultraflat d=3} of Theorem \ref{thm:UpperBound3d} follows along the same lines as the above, this time using the second part of Lemma \ref{lem:LowerBoundPairs} in place of the first one.
\end{proof}
Note that in the proof of of the lower bound in Theorem \ref{thm:UpperBound2d} we have
used the abundance of close-by pairs of lattice points with \textbf{$\left|c_{\lambda}\right|^{2}\ge \frac{1}{2N}$};
in the absence of such close-by lattice points, the bound does not hold. For example, for $d=2$, fix $\lambda_{0}\in\mathcal{E}_{n}$
and let $\left|c_{\pm\lambda_{0}}\right|^{2}=1/2$ and $c_{\lambda}=0$
for every $\lambda\ne\pm\lambda_{0}.$ Then
\[
\Vc(X_{f_{n},r})=4\pi^{2}r^{4}h_{2}\left(2T\right)\ll r^{4}T^{-3}.
\]

\begin{proof}[Proof of Lemma \ref{lem:LowerBoundPairs}]
	For the first part of Lemma \ref{lem:LowerBoundPairs}, divide $S^{1}$ into $k=O\left(T\right)$
	arcs $I_{1},I_{2},\dots,I_{k}$ of length $<1/T$. For every $1\le j\le k,$
	let $n_{j}=\#\left\{ m:\,x_{m}\in I_{j}\right\} ,$ so $\sum_{j=1}^{k}n_{j}=M$.
	By the Cauchy-Schwarz inequality,
	\[
	M^{2}=\left(\sum_{j=1}^{k}n_{j}\right)^{2}\le k\sum_{j=1}^{k}n_{j}^{2}\ll T\sum_{j=1}^{k}n_{j}^{2}.
	\]
	Thus,
	\begin{align*}
	\#\left\{ x_{i}\ne x_{j}:\,\left|x_{i}-x_{j}\right|\le1/T\right\}  & =\#\left\{ x_{i},x_{j}:\,\left|x_{i}-x_{j}\right|\le1/T\right\} -M\\
	& \gg\sum_{j=1}^{k}n_{j}^{2}-M\gg M^{2}/T-M\gg M^{2}/T.
	\end{align*}
	The second part of Lemma \ref{lem:LowerBoundPairs} is proved similarly.
\end{proof}

\section{Restricted averages}

\label{sec:RestrictedAverages}

\subsection{Restricted moments}

For $ d=2 $, most of our principal results above are also valid in the more difficult scenario where $x$ is drawn in $B_{x_{0}}(\rho)$  for some
$x_{0}\in\Tb^{2}$ and $\rho\gg n^{-1/2+o(1)}$. In this case, the restricted moments are: expectation
\begin{equation}
\label{eq:restricted_expectation}
\E_{x_{0},\rho}[X_{f_{n},r}] = \frac{1}{\vol(B_{x_0}(\rho))}\int\limits_{B_{x_{0}}(\rho)}X_{f_{n},r;x}dx,
\end{equation}
higher centred moments
\begin{equation}
\label{eq:centred moments rest}
\E_{x_{0},\rho}[(X_{f_{n},r}-\E_{x_{0},\rho}[X_{f_{n},r}])^{k}] = \frac{1}{\vol(B_{x_0}(\rho))}\int\limits_{B_{x_{0}}(\rho)}\left(X_{f_{n},r;x}-\E_{x_{0},\rho}[X_{f_{n},r}]\right)^{k}dx, \hspace{1em}k\ge2,
\end{equation}
and in particular the variance
\begin{equation}
\label{eq:restricted_variance}
\Vc_{x_{0},\rho}(X_{f_{n},r}) = \E_{x_{0},\rho}[(X_{f_{n},r}-\E_{x_{0},\rho}[X_{f_{n},r}])^{2}].
\end{equation}

We reinterpret the statement of Granville-Wigman's ~\cite[Theorem 1.2]{GranvilleWigman} as evaluating the expected mass
$$\E_{x_{0},\rho}[X_{f_{n},r}] \sim \pi r^{2},$$ valid for almost all $n\in S_{2}$,  uniformly for  $\rho\gg n^{-1/2+o(1)}$, $x_{0}\in\Tb^{2}$,
and $ r>0 $ (see the first part of Lemma \ref{lem:ExpVarShrinking}).

\subsection{Quasi-correlations}

For the restricted moments \eqref{eq:centred moments rest}
of $X_{f_{n},r}$ one also needs to cope with {\em quasi-correlations}, i.e. tuples $(\lambda^{1},\ldots,\lambda^{l}) \in \Ec_{n}^{l}$
with the sum $\sum\limits_{i=1}^{l}\lambda^{i}$ unexpectedly small, e.g. given a (small) fixed number $\delta >0$,
\begin{equation}
\label{eq:sum tuple small l,delta}
\left\|\sum\limits_{i=1}^{l}\lambda^{i}\right\| < n^{1/2-\delta};
\end{equation}
unlike the correlations \eqref{eq:Sc correlations def},
here there are no congruence obstructions, so that \eqref{eq:sum tuple small l,delta} makes sense
with $l$ odd or even.

\begin{definition}[Quasi-correlations, cf. {~\cite[Definition 1.3]{BMW}}]
\label{def:separatedness Ac} \hspace{1mm}

\begin{enumerate}

\item For $n\in S_2$, $l\in\mathbb{Z}_{\ge 2}$, and $0 < K=K(n) < l\cdot n^{1/2}$  define the set of length-$l$ spectral quasi-correlations
\begin{equation}
\label{eq:quasi_correlations}
\mathcal{C}_{n}(l;K) =
\left\{(\lambda^{1},\ldots,\lambda^{l})\in\mathcal{E}_{n}^{l}:\: 0 < \left\| \sum\limits_{j=1}^{l}\lambda^{j} \right\| \le K  \right\}.
\end{equation}

\item Given $\delta>0$ we say that $n \in S_{2}$ satisfies the $(l,\delta)$-separateness hypothesis $\mathcal{A}(n;l,\delta)$ if
\begin{equation}
\label{eq:sep_hypothesis}
\mathcal{C}_{n}(l;n^{1/2-\delta}) = \varnothing.
\end{equation}

\end{enumerate}

\end{definition}

For example, $\Ac(n;2,\delta)$ is equivalent to the aforementioned Bourgain-Rudnick separateness, satisfied ~\cite[Lemma 5]{Bourgain-Rudnick} by a density $1$ sequence $S_{2}'\subseteq S_{2}$. More generally, it was shown in the forthcoming paper ~\cite{BenatarBuckleyWigman},
that for every $\delta>0$ and $l\ge 2$, the assumption $\Ac(n;l,\delta)$ is satisfied by generic $n\in S_{2}'(l,\delta)$, and hence
a standard diagonal argument yields a density $1$ sequence $S_{2}'\subseteq S_{2}$ so that $\Ac(n;l,\delta)$ is satisfied
for {\em all} $l\ge 2$ and $ \delta>0 $ for $ n\in S_2' $ sufficiently large.

\begin{theorem}[\cite{BenatarBuckleyWigman}]
\label{thm:quasi-corr small}
For every fixed $l\ge 2$ and $\delta>0$ there exist a set $S_2'(l,\delta)\subseteq S_2$ such that:

\begin{enumerate}

\item The set $S_2'(l,\delta)$ has density $1$ in $S_2$.

\item For every $n\in S_2'(l,\delta)$ the length-$l$
spectral quasi-correlation set \[\mathcal{C}_{n}(l;n^{1/2-\delta})=\varnothing\] is empty, i.e., the $(l,\delta)$-separateness hypothesis $\mathcal{A}(n;l,\delta)$ is satisfied.

\end{enumerate}

\end{theorem}


\subsection{A version of Theorem \ref{thm:VarMainGeneralized} with restricted averages}

We have the following analogue of Theorem \ref{thm:VarMainGeneralized}:

\begin{theorem}
	\label{thm:VarMainExplRestricted}
	
	\begin{enumerate}  Let $\delta>0$, and $0<\epsilon< \delta/5$.
		
		\item If $S_{2}'\subseteq S_{2}$ is a sequence satisfying the hypotheses $ \mathcal{D}(n,\epsilon/2),$ $ \Ac(n;2,\epsilon)$, and $\Ac(n;4,\epsilon)$ for all $n\in S_{2}'$, then in the setting of Theorem \ref{thm:VarMainGeneralized} part (1),
		\begin{equation*}
		\Vc_{x_{0},\rho}\left(X_{f_{n},r}\right)\sim\frac{16}{3\pi \cos^{2}\theta_{f_{n}}}r^{4}T^{-1}
		\end{equation*}
		uniformly for all $x_{0}\in\mathbb{T}^{2}$, $n^{-1/2+\delta}\le\rho\le 1$
		and $r_{0}< r< r_{1}$, and $f_{n}\in\Fc_{1}(n;T(n),\eta(n))$.
		
		\item Let $k\ge 3$ be an integer. If $S_{2}'\subseteq S_{2}$ is a sequence satisfying the length-$2k$ diagonal domination
		assumption and the hypotheses $ \mathcal{D}(n,\epsilon),$ $ \Ac(n;2,\epsilon)$, $ \Ac(n;4,\epsilon)$, and $\Ac(n;2k,\epsilon)$ for all $n\in S_{2}'$, then in the setting of Theorem \ref{thm:VarMainGeneralized} part (2),
		\begin{equation*}
		\mathbb{E}_{x_{0},\rho}\left[\hat{X}_{f_{n},r}^{k}\right] \to \E[Z^{k}]
		\end{equation*}
		uniformly for  $x_{0}\in\Tb^{2}$, $r_{0} < r <r_{1}$, $n^{-1/2+\delta} \le \rho \le 1$, and $f_{n}\in \Fc_{2}(n;T(n),\eta(n))$,
		where $Z\sim N(0,1)$ is the standard Gaussian variable.
		
	\end{enumerate}
	
\end{theorem}

Theorem \ref{thm:VarMainExplRestricted} follows along similar lines as the proof of Theorem \ref{thm:VarMainGeneralized}, where we use the expressions for the restricted moments below (cf. equation \eqref{eq:ExpectationEquality}, Lemma \ref{lem:VarExpd2} and equation \eqref{eq:kthMoment}). We remark that Theorem \ref{thm:UpperBound2d} can also be extended to $ \Vc_{x_{0},\rho}(X_{f_{n},r}) $, however the lower bound will only hold for a generic $ n\in S_2 $.

\begin{lemma}[Expectation and variance, $d=2$, $x$ drawn in shrinking discs]
\label{lem:ExpVarShrinking}

For $d=2$ let $0<\delta<1/2$, $0<\epsilon<\delta/5$, and $S_{2}'\subseteq S_{2}$.

\begin{enumerate}

\item

If  $n\in S_{2}'$ satisfy the hypothesis $\Ac(n;2,\epsilon) $, then
	\[
	\mathbb{E}_{x_{0},\rho}\left[X_{f_{n},r}\right]=\pi r^{2}+O\left(r^{2}n^{-\frac{3}{5}\delta+3\epsilon}\right)
	\]
	uniformly for $x_{0}\in\mathbb{T}^{2},$ $n^{-1/2+\delta}\le\rho\le1$
	and $r>0$.

\item
\label{lem:VarFormulaShrinking}

If  $n\in S_{2}'$ satisfy the hypotheses $ \Ac(n;2,\epsilon) $ and $\Ac(n;4,\epsilon)$, then
	\begin{align*}
	\text{\ensuremath{\mathcal{V}}}_{x_{0},\rho}\left(X_{f_{n},r}\right) & =8\pi^{2}r^{4}\sum_{\begin{subarray}{c}
		\lambda,\lambda'\in\mathcal{E}_{n}\\
		\lambda\ne\lambda'
		\end{subarray}}\left|c_{\lambda}\right|^{2}\left|c_{\lambda'}\right|^{2}h_{2}\left(r\left\Vert \lambda-\lambda'\right\Vert \right) +O\left(r^{4} n^{-\frac{3}{5}\delta+4\epsilon}\right)
	\end{align*}
uniformly for $x_{0}\in\mathbb{T}^{2},$ $n^{-1/2+\delta}\le\rho\le1$ and $r>0$.

\end{enumerate}

\end{lemma}

\begin{lemma}[Higher moments, $d=2$, $x$ drawn in shrinking discs]
	\label{lem:KthMoment}For $d=2$ let $ k\ge3 $, $0<\delta<1/2$,
	$0<\epsilon< \delta/5$, and $ S_2' \subseteq S_2 $ satisfying $ \Ac(2;n,\epsilon), $ $ \Ac(4;n,\epsilon), $ and $\Ac(n;2k,\epsilon) $ for every $ n\in S_2' $ . We have
	\begin{align*}
	\E_{x_{0},\rho}[\hat{X}_{f_{n},r}^{k}] & =(2\pi)^k r^{2k}\Vc_{x_{0},\rho}\left(X_{f_{n},r}\right)^{-k/2}\sum_{\begin{subarray}{c}
		\forall1\le i\le k,\,\lambda_{i}\ne\lambda_{i}'\in\mathcal{E}_{n}\\
		\sum_{i=1}^{k}\left(\lambda_{i}-\lambda_{i}'\right)=0
		\end{subarray}}\prod_{j=1}^{k}c_{\lambda_{j}}\overline{c_{\lambda_{j}'}}g_{2}\left(r\left\Vert \lambda_{j}-\lambda_{j}'\right\Vert \right)\\
	& +O\left(\Vc_{x_{0},\rho}\left(X_{f_{n},r}\right)^{-k/2}  r^{2k} n^{-\frac{3}{5}\delta+4\epsilon}\right)\nonumber
	\end{align*}
	uniformly for $x_{0}\in\mathbb{T}^{2},$ $n^{-1/2+\delta}\le\rho\le1$
	and $r>0$.
\end{lemma}

\subsection{Proofs of Lemma \ref{lem:ExpVarShrinking} and Lemma \ref{lem:KthMoment}}

\begin{proof}[Proof of Lemma \ref{lem:ExpVarShrinking}, $1$st part]
	We have
	\begin{equation}
	\label{eq:exp_basic_formula}
	\E_{x_{0},\rho}\left[X_{f_{n},r}\right]=\frac{1}{\pi\rho^{2}}\int_{B_{x_{0}}\left(\rho\right)}\int_{B_{x}\left(r\right)}f_{n}\left(y\right)^{2}\,\text{d}y\,\text{d}x.
	\end{equation}
	
	Granville-Wigman's
	\cite[Theorem 1.2]{GranvilleWigman} asserts that for $ \epsilon_1 > \epsilon_2 > 0$, $ 0 < \epsilon_3 < \epsilon_1 - \epsilon_2 $ and  $ n\in S_2 $ satisfying  $\Ac(n;2,\epsilon_2) $, we have
	\begin{equation}
	\label{eq:Granville_Wigman_theorem}
	\int_{B_{x}\left(r\right)}f_{n}\left(y\right)^{2}\,\text{d}y = \pi r^2 \left(1+ O\left(n^{-3\epsilon_3 /2}\right)\right)
	\end{equation}
	uniformly in $ x\in \mathbb{T}^2 $ and $ r>n^{-1/2 + \epsilon_1} $. If $r>n^{-1/2+\frac{2}{5}\delta}$, then by substituting \eqref{eq:Granville_Wigman_theorem} with  $ \epsilon_1 = \frac{2}{5}\delta $, $ \epsilon_2 = \epsilon $ and $ \epsilon_3 = \frac{2}{5}\delta - 2\epsilon $ into \eqref{eq:exp_basic_formula}, we have \[\E_{x_{0},\rho}\left[X_{f_{n},r}\right]=\pi r^{2}\left(1+O\left(n^{-\frac{3}{2}\left( \frac{2}{5}\delta - 2\epsilon\right)}\right)\right)\]
	for every $\rho$.
	
	Otherwise, note that
	\[
	\E_{x_{0},\rho}\left[X_{f_{n},r}\right]=\frac{1}{\pi\rho^{2}}\int_{B_{x_{0}}\left(\rho+r\right)}f_{n}\left(y\right)^{2}\int_{B_{x_{0}}\left(\rho\right)\cap B_{y}\left(r\right)}\,\text{d}x\,\text{d}y,
	\]
	so
	\begin{equation}
	\label{eq:Exp_upper_lower_bnds}
	\frac{r^{2}}{\rho^{2}}\int_{B_{x_{0}}\left(\rho-r\right)}f_{n}\left(y\right)^{2}\,\text{d}y\le\E_{x_{0},\rho}\left[X_{f_{n},r}\right]\le\frac{r^{2}}{\rho^{2}}\int_{B_{x_{0}}\left(\rho+r\right)}f_{n}\left(y\right)^{2}\,\text{d}y.
	\end{equation}
	Since $r/\rho \le n^{-\frac{3}{5}\delta}$, we can use \eqref{eq:Granville_Wigman_theorem} with  $ \epsilon_1 = \delta $, $ \epsilon_2 = \epsilon $ and $ \epsilon_3 = \delta - 2\epsilon $
	to deduce that
	\begin{equation}
	\label{eq:Exp_Inner_integral}
	\int_{B_{x_{0}}\left(\rho\pm r\right)}f_{n}\left(y\right)^{2}\,\text{d}y=\pi\rho^{2}\left(1+O\left(n^{-\frac{3}{5}\delta}\right)\right),
	\end{equation} and the statement of the first part of Lemma
	\ref{lem:ExpVarShrinking} follows upon substituting \eqref{eq:Exp_Inner_integral} into \eqref{eq:Exp_upper_lower_bnds}.
\end{proof}
\begin{proof}[Proof of Lemma \ref{lem:ExpVarShrinking}, $2$nd part]
	We have
	\begin{equation*}
	\Vc_{x_{0},\rho}(X_{f_{n},r}) = \frac{1}{\pi\rho^{2}}\int_{B_{x_{0}}\left(\rho\right)}\left(\int_{B_{x}\left(r\right)}f_{n}\left(y\right)^{2}\,\text{d}y-	\E_{x_{0},\rho}\left[X_{f_{n},r}\right]\right)^{2}\,\text{d}x.
	\end{equation*}
	
	By (\ref{eq:IntegrandVar}),
	\begin{align*}
	& \frac{1}{\pi\rho^{2}}\int_{B_{x_{0}}\left(\rho\right)}\left(\int_{B_{x}\left(r\right)}f_{n}\left(y\right)^{2}\,\text{d}y-\pi r^{2}\right)^{2}\,\text{d}x\\
	& =4\pi^{2}r^{4}\sum_{\begin{subarray}{c}
		\lambda,\lambda',\lambda'',\lambda'''\in\mathcal{E}_{n}\\
		\lambda\ne\lambda'\\
		\lambda''\ne\lambda'''
		\end{subarray}}c_{\lambda}\overline{c_{\lambda'}}c_{\lambda''}\overline{c_{\lambda'''}}g_{2}\left(r\left\Vert \lambda-\lambda'\right\Vert \right)g_{2}\left(r\left\Vert \lambda''-\lambda'''\right\Vert \right)\\
	& \times\frac{1}{\pi\rho^{2}}\int_{B_{x_{0}}\left(\rho\right)}e\left(\left\langle x,\lambda-\lambda'+\lambda''-\lambda'''\right\rangle \right)\,\text{d}x\\
	& =8\pi^{2}r^{4}\sum_{\begin{subarray}{c}
		\lambda,\lambda'\in\mathcal{E}_{n}\\
		\lambda\ne\lambda'
		\end{subarray}}\left|c_{\lambda}\right|^{2}\left|c_{\lambda'}\right|^{2}g_{2}\left(r\left\Vert \lambda-\lambda'\right\Vert \right)^{2}\\
	& +8\pi^{2}r^{4}\sum_{\begin{subarray}{c}
		\lambda,\lambda',\lambda'',\lambda'''\in\mathcal{E}_{n}\\
		\lambda\ne\lambda'\\
		\lambda''\ne\lambda'''\\
		\lambda-\lambda'+\lambda''-\lambda'''\ne0
		\end{subarray}}c_{\lambda}\overline{c_{\lambda'}}c_{\lambda''}\overline{c_{\lambda'''}}g_{2}\left(r\left\Vert \lambda-\lambda'\right\Vert \right)g_{2}\left(r\left\Vert \lambda''-\lambda'''\right\Vert \right)\\
	& \times e\left(\left\langle x_{0},\lambda-\lambda'+\lambda''-\lambda'''\right\rangle \right)g_{2}\left(\rho\left\Vert \lambda-\lambda'+\lambda''-\lambda'''\right\Vert \right).
	\end{align*}
	By the hypothesis $\Ac(n;4,\epsilon)$ and Lemma \ref{lem:H2Formulas}, we have
	\begin{align*}
	& \sum_{\begin{subarray}{c}
		\lambda,\lambda',\lambda'',\lambda'''\in\mathcal{E}_{n}\\
		\lambda\ne\lambda'\\
		\lambda''\ne\lambda'''\\
		\lambda-\lambda'+\lambda''-\lambda'''\ne0
		\end{subarray}}c_{\lambda}\overline{c_{\lambda'}}c_{\lambda''}\overline{c_{\lambda'''}}g_{2}\left(r\left\Vert \lambda-\lambda'\right\Vert \right)g_{2}\left(r\left\Vert \lambda''-\lambda'''\right\Vert \right)\\
	& \times e\left(\left\langle x_{0},\lambda-\lambda'+\lambda''-\lambda'''\right\rangle \right)g_{2}\left(\rho\left\Vert \lambda-\lambda'+\lambda''-\lambda'''\right\Vert \right)\\
	& \ll\left(\sum_{\lambda\in\mathcal{E}_{n}}\left|c_{\lambda}\right|\right)^{4}\frac{1}{\left(n^{\delta-\epsilon}\right)^{3/2}}\ll N^{2}n^{-\frac{3}{2}(\delta-\epsilon)} \ll n^{-\frac{3}{2}\delta + 2\epsilon}.
	\end{align*}
	
	Next, note that
	\begin{align}
	\label{eq:inner_int_estimate}
	\int_{B_{x}\left(r\right)}f_{n}\left(y\right)^{2}\,\text{d}y-\pi r^{2} &=  2\pi r^{2}\sum_{\lambda\ne\lambda'\in\mathcal{E}_{n}}
	c_{\lambda}\overline{c_{\lambda'}}g_{2}\left(r\left\Vert \lambda-\lambda'\right\Vert \right)   \ll r^2 \left(\sum_{\lambda\in\mathcal{E}_{n}}\left|c_{\lambda}\right|\right)^{2} \ll N r^2.
	\end{align}	
	By \eqref{eq:inner_int_estimate} and the first part of Lemma \ref{lem:ExpVarShrinking},
	
	\begin{equation*}
	\Vc_{x_{0},\rho}(X_{f_{n},r}) = \frac{1}{\pi\rho^{2}}\int_{B_{x_{0}}\left(\rho\right)}\left(\int_{B_{x}\left(r\right)}f_{n}\left(y\right)^{2}\,\text{d}y-\pi r^{2}\right)^{2}\,\text{d}x + O\left(r^4 n^{-\frac{3}{5}\delta+4\epsilon}\right)
	\end{equation*}	
	and the statement of Lemma \ref{lem:VarExpd2} follows.
\end{proof}

\begin{proof}[Proof of Lemma \ref{lem:KthMoment}]
	We have
	\begin{equation*}
	\E_{x_{0},\rho}[\hat{X}_{f_{n},r}^{k}]  = \Vc_{x_{0},\rho}\left(X_{f_{n},r}\right)^{-k/2} \cdot \frac{1}{\pi\rho^{2}}\int_{B_{x_{0}}\left(\rho\right)}\left(\int_{B_{x}\left(r\right)}f_{n}\left(y\right)^{2}\,\text{d}y-	\E_{x_{0},\rho}\left[X_{f_{n},r}\right]\right)^{k}\,\text{d}x.
	\end{equation*}
	By (\ref{eq:IntegrandVar}), we have
	\begin{align*}
	\frac{1}{\pi\rho^{2}}&\int_{B_{x_{0}}\left(\rho\right)} \left(\int_{B_{x}\left(r\right)}f_{n}\left(y\right)^{2}\,\text{d}y-\pi r^{2}\right)^{k}\,\text{d}x = \left(2\pi\right)^{k}r^{2k} \sum_{\begin{subarray}{c}
		\forall1\le i\le k,\,\lambda_{i}\ne\lambda_{i}'\in\mathcal{E}_{n}\\
		\sum_{i=1}^{k}\left(\lambda_{i}-\lambda_{i}'\right)=0
		\end{subarray}}\prod_{j=1}^{k}c_{\lambda_{j}}\overline{c_{\lambda_{j}'}}g_{2}\left(r\left\Vert \lambda_{j}-\lambda_{j}'\right\Vert \right)\\
	& +\left(2\pi\right)^{k}r^{2k}\sum_{\begin{subarray}{c}
		\forall1\le i\le k,\,\lambda_{i}\ne\lambda_{i}'\in\mathcal{E}_{n}\\
		\sum_{i=1}^{k}\left(\lambda_{i}-\lambda_{i}'\right)\ne0
		\end{subarray}}\prod_{j=1}^{k}c_{\lambda_{j}}\overline{c_{\lambda_{j}'}}g_{2}\left(r\left\Vert \lambda_{j}-\lambda_{j}'\right\Vert \right)\\
	& \times2e\left(\left\langle x_{0},\sum_{j=1}^{k}\left(\lambda_{j}-\lambda_{j}'\right)\right\rangle \right)g_{2}\left(\rho\left\Vert \sum_{j=1}^{k}\left(\lambda_{j}-\lambda_{j}'\right)\right\Vert \right).
	\end{align*}
	By the hypothesis $\Ac(n;2k,\epsilon)$,
	\begin{align*}
	\sum_{\begin{subarray}{c}
		\forall1\le i\le k,\,\lambda_{i}\ne\lambda_{i}'\in\mathcal{E}_{n}\\
		\sum_{i=1}^{k}\left(\lambda_{i}-\lambda_{i}'\right)\neq0
		\end{subarray}}& \prod_{j=1}^{k}c_{\lambda_{j}}\overline{c_{\lambda_{j}'}}g_{2}\left(r\left\Vert \lambda_{j}-\lambda_{j}'\right\Vert \right) e\left(\left\langle x_{0},\sum_{j=1}^{k}\left(\lambda_{j}-\lambda_{j}'\right)\right\rangle \right)g_{2}\left(\rho\left\Vert \sum_{j=1}^{k}\left(\lambda_{j}-\lambda_{j}'\right)\right\Vert \right)\\
	& \ll\left(\sum_{\lambda\in\mathcal{E}_{n}}\left|c_{\lambda}\right|\right)^{2k}\frac{1}{\left(n^{\delta - \epsilon}\right)^{3/2}}\ll N^k n^{-\frac{3}{2}(\delta-\epsilon)} \ll n^{-\frac{3}{2}+2\epsilon}.
	\end{align*}
	By \eqref{eq:inner_int_estimate} and the first part of Lemma \ref{lem:ExpVarShrinking},
	\begin{align*}
	\mathbb{E}_{x_{0},\rho}[\hat{X}_{f_{n},r}^{k}] & = \Vc_{x_{0},\rho}\left(X_{f_{n},r}\right)^{-k/2} \cdot \frac{1}{\pi\rho^{2}}\int_{B_{x_{0}}\left(\rho\right)}\left(\int_{B_{x}\left(r\right)}f_{n}\left(y\right)^{2}\,\text{d}y-\pi r^2 \right)^{k}\,\text{d}x \\ &+  O\left(\Vc_{x_{0},\rho}\left(X_{f_{n},r}\right)^{-k/2} r^{2k}  n^{-\frac{3}{5}\delta+4\epsilon}\right),
	\end{align*}
	and the statement of Lemma \ref{lem:KthMoment} follows.
\end{proof}

\section{\label{sec:AuxLemmasProof}Proofs of auxiliary lemmas}

In this section we provide the proofs for lemmas \ref{lem:BasicVarProp}, \ref{lem:InnerIntegral}, and \ref{lem:CosToDist}:
\begin{proof}[Proof of Lemma \ref{lem:BasicVarProp}]
	~
	\begin{enumerate}
		\item The upper bound is straightforward, and the lower bound follows from (\ref{eq:BasicNormalization})
		by invoking the Cauchy-Schwarz inequality on \eqref{eq:BasicNormalization}.
		\item By partial summation, for every $\lambda_{0}\in\mathcal{E}_{n}$ we
		have
		\[
		1=\sum_{\lambda\in\mathcal{E}_{n}}\left|c_{\lambda}\right|^{2}=N\left|c_{\lambda_{0}}\right|^{2}+E,
		\]
		where $\left|E\right|\le V\left(\underline{v} \right)$.
		Since $\lambda_{0}$ is arbitrary, we deduce that
		\[
		[\underline{v}]_{\infty} \le1+V\left(\underline{v} \right).
		\]
		\item Follows directly from parts $1$ and $2$ of this lemma.
	\end{enumerate}
\end{proof}

\begin{proof}[Proof of Lemma \ref{lem:InnerIntegral}]
	We have
	\begin{align}
		\label{eq:inner_int_expansion}
		\int_{B_{x}\left(r\right)}f_{n}\left(y\right)^{2}\,\text{d}y&=\int_{B_{x}\left(r\right)}\sum_{\lambda,\lambda'\in\mathcal{E}_{n}}c_{\lambda}\overline{c_{\lambda'}}e\left(\left\langle y,\lambda-\lambda'\right\rangle \right)\,\text{d}y\\
		& =\frac{\pi^{d/2}}{\Gamma\left(d/2+1\right)}r^{d}+\sum_{\begin{subarray}{c}
				\lambda,\lambda'\in\mathcal{E}_{n}\\
				\lambda\ne\lambda'
		\end{subarray}}c_{\lambda}\overline{c_{\lambda'}}\int_{B_{x}\left(r\right)}e\left(\left\langle y,\lambda-\lambda'\right\rangle \right)\,\text{d}y.\nonumber
	\end{align}
	Transforming the variables $y=rz+x$, we obtain
	\begin{equation}
		\label{eq:var_transformation}
		\int_{B_{x}\left(r\right)}e\left(\left\langle y,\lambda-\lambda'\right\rangle \right)\,\text{d}y=r^{d}e\left(\left\langle x,\lambda-\lambda'\right\rangle \right)\int_{B_{0}\left(1\right)}e\left(\left\langle z,r\left(\lambda-\lambda'\right)\right\rangle \right)\,\text{d}z.
	\end{equation}
	Note that
	\begin{align}
		\label{eq:Fourier_ball}
		\int_{B_{0}\left(1\right)}e\left(\left\langle z,r\left(\lambda-\lambda'\right)\right\rangle \right)\,\text{d}z & =\frac{\left(2\pi\right)^{d/2}J_{d/2}\left(2 \pi r\left\Vert \lambda-\lambda'\right\Vert \right)}{\left(2 \pi r\left\Vert \lambda-\lambda'\right\Vert \right)^{d/2}},
	\end{align}
	and (\ref{eq:IntegrandVar}) follows upon substituting \eqref{eq:Fourier_ball} into \eqref{eq:var_transformation} and finally into \eqref{eq:inner_int_expansion}.
\end{proof}

\begin{proof}[Proof of Lemma \ref{lem:CosToDist}]
	Let $\theta_{\lambda}$ be the angle between $\lambda$ and $\lambda'$.
	Then
	\begin{align*}
	\frac{1}{N} \cdot \#\left\{ \lambda\in\mathcal{E}_{n}:\,\lambda\succeq\lambda',\left\Vert \widehat{\lambda}-\widehat{\lambda'}\right\Vert \le s\right\}
	& =\frac{1}{N}\cdot \#\left\{ \lambda\in\mathcal{E}_{n}:\,\theta_{\lambda}\ge0,\,\sqrt{2\left(1-\cos\theta_{\lambda}\right)}\le s\right\} \\
	& =\frac{1}{N}\cdot \#\left\{ \lambda\in\mathcal{E}_{n}:\,\theta_{\lambda}\in\left[0,\arccos\left(1-s^{2}/2\right)\right]\right\} \\
	& =\frac{1}{2\pi}\arccos\left(1-s^{2}/2\right)+O\left(\Delta\left(n\right)\right)\\
	& =\frac{s}{2\pi}+O\left(s^{3}+\Delta\left(n\right)\right)
	\end{align*}
	which is the statement (\ref{eq:CosToDistEq}) of Lemma \ref{lem:CosToDist}.
\end{proof}

\end{document}